\documentclass{amsart}

\usepackage[mathscr]{euscript}
\usepackage{amsfonts,amssymb}
\usepackage[cmtip,arrow]{xy}
\usepackage{pb-diagram,pb-xy}
\usepackage{hyperref}

\theoremstyle{plain}
\newtheorem{theorem}[equation]{Theorem}
\newtheorem{lemma}[equation]{Lemma}
\newtheorem{corollary}[equation]{Corollary}
\newtheorem{proposition}[equation]{Proposition}
\newtheorem{problem}[equation]{Problem}

\theoremstyle{definition}
\newtheorem{definition}[equation]{Definition}
\newtheorem{remark}[equation]{Remark}
\newtheorem{example}[equation]{Example}
\newtheorem{examples}[equation]{Examples}

\begin{document}

\title{Goodwillie Calculus}
\author{Gregory Arone and Michael Ching}

\maketitle

\tableofcontents

\numberwithin{equation}{section}
\renewcommand{\theequation}{\thesection.\arabic{equation}}

Goodwillie calculus is a method for analyzing functors that arise in topology. One may think of this theory as a categorification of the classical differential calculus of Newton and Leibnitz, and it was introduced by Tom Goodwillie in a series of foundational papers~\cite{goodwillie:1990,goodwillie:1991,goodwillie:2003}.

The starting point for the theory is the concept of an \emph{$n$-excisive} functor, which is a categorification of the notion of a polynomial function of degree $n$.  One of Goodwillie's key results says that every homotopy functor $F$ has a universal approximation by an $n$-excisive functor $P_nF$, which plays the role of the $n$-th Taylor approximation of $F$. Together, the functors $P_nF$ fit into a tower of approximations of $F$: the \emph{Taylor tower}
\[ F\longrightarrow \cdots \longrightarrow P_nF  \longrightarrow \cdots\longrightarrow P_1F \longrightarrow P_0F \]

It turns out that $1$-excisive functors are the ones that represent generalized homology theories (roughly speaking). For example, if $F=I$ is the identity functor on the category of based spaces, then $P_1I$ is the functor $P_1I(X)\simeq \Omega^\infty \Sigma^\infty X$. This functor represents stable homotopy theory in the sense that $\pi_*(P_1I(X))\cong \pi_*^s(X)$. Informally, this means that the best approximation to the homotopy groups by a generalized homology theory is given by the stable homotopy groups. The Taylor tower of the identity functor then provides a sequence of theories, satisfying higher versions of the excision axiom, that interpolate between stable and unstable homotopy.

The analogy between Goodwillie calculus and ordinary calculus reaches a surprising depth. To illustrate this, let $D_nF$ be the homotopy fiber of the map $P_nF \to P_{n-1}F$. The functors $D_nF$ are the homogeneous pieces of the Taylor tower. They are controlled by Taylor ``coefficients'' or derivatives of $F$. This means that for each $n$ there is a spectrum with an action of $\Sigma_n$ that we denote $\partial_nF$, and there is an equivalence of functors
\[ D_nF(X) \simeq \Omega^\infty \left(\partial_n F\wedge X^{\wedge n}\right)_{h\Sigma_n}. \]
Here for concreteness $F$ is a homotopy functor from the category of pointed spaces to itself; similar formulas apply for functors to and from other categories. The spectrum $\partial_n F$ plays the role of the $n$-th derivative of $F$, and the spectra $\partial_n F$ are relatively easy to calculate. There is an obvious similarity between the formula for $D_nF$ and the classical formula for the $n$-th term of the Taylor series of a function.

Of course there are differences between the classical differential calculus and Goodwillie calculus, and they are just as interesting as the similarities. One place where the analogy breaks down is in the complex ways that the homogeneous pieces can be ``added up'' to create the full Taylor tower. Homogeneous layers only determine the Taylor tower up to extensions. Theorems of Randy McCarthy, Nick Kuhn and the present authors reveal situations in which these extension problems can be understood via interesting connections to Tate spectra. Considerable simplification occurs by passing to chromatic homotopy theory, a fact that forms the basis for recent work of Gijs Heuts on the classification of unstable $v_n$-periodic homotopy theory via spectral Lie algebras.

Also unlike ordinary calculus, a crucial example is provided by the identity functor (for based spaces or, more generally, for any $\infty$-category of interest). As hinted at above, the identity functor typically has an interesting and non-trivial Taylor tower controlled by its own sequence of Taylor coefficients. For based spaces, these derivatives were first calculated by Brenda Johnson. Mark Mahowald and the first author used a detailed description of these objects to get further information about the Taylor tower of the identity in unstable $v_n$-periodic homotopy. The derivatives of the identity also play an important theoretical role in the calculus; in particular, by a result of the second author they form an operad that encodes structure possessed by the derivatives of any functor to or from a given $\infty$-category. For topological spaces, this operad is a topological analogue of the Lie operad, explaining the role of Lie algebras in Quillen's work on rational homotopy theory, and in Heuts's work mentioned above. Structures related to the Lie operad also form the basis of a classification of functors up to Taylor tower equivalence given by the present authors.

The nature of Goodwillie calculus lends itself to both computational and conceptual applications. Goodwillie originally developed the subject in order to understand more systematically certain calculations in algebraic $K$-theory, and this area remains a compelling source of specific examples. However, the deeply universal nature of these constructions gives functor calculus a crucial role in the foundations of homotopy theory, especially given the expanding role therein of higher category theory.

Indeed it seems that the calculus has not yet found its most general form. The similarities to Goodwillie calculus borne by the manifold and orthogonal ``calculi'' of Michael Weiss suggest some deeper structure that is still to be properly worked out. There are also important but not fully understood connections to manifolds and factorization homology of $E_n$-algebras, and a properly equivariant version of Goodwillie calculus has been hinted at by work of Dotto, but much remains to be explored.

Notwithstanding such future developments, the fundamental role of Goodwillie calculus in homotopy theory is as clear as that of ordinary calculus in other areas of mathematics: it provides a systematic interpolation between the linear (homotopy-theoretically, this usually means the stable) and nonlinear (or unstable) worlds, and thus brings our deep intuition of the nature of change to bear on our understanding of homotopy theory.

\subsection*{Acknowledgements}

The first author was partially supported by the Swedish Research Council (grant number 2016-05440), and the second author was partially supported by the National Science Foundation (grant DMS-1709032) and the Isaac Newton Institute (EPSRC grant number EP/R014604/1 and the HHH programme), during the writing of this article. Useful feedback and corrections to an earlier draft were provided by Gijs Heuts and Haynes Miller.

\section{Polynomial Approximation and the Taylor Tower}

Goodwillie's calculus of functors is modelled after ordinary differential calculus with the role of smooth maps between manifolds played by \emph{homotopy functors} $F: \mathscr{C} \to \mathscr{D}$, i.e. those that preserve some notion of weak equivalence. In Goodwillie's original formulation \cite{goodwillie:1990}, the categories $\mathscr{C}$ and $\mathscr{D}$ were each taken to be some category of topological spaces or spectra, but modern higher-category-theoretic technology allows the theory to be developed for functors $F$ between any $(\infty,1)$-categories that are suitably well-behaved.

The basic tenets of the theory are independent of any particular model for $(\infty,1)$-categories. In this paper, we will mostly use the language of $\infty$-categories (i.e. quasicategories) from \cite{lurie:2009}, and the details of Goodwillie calculus have been developed in the greatest generality by Lurie in that context, see \cite[Sec. 6]{lurie:2017}. Thus our typical assumption will be that $F: \mathscr{C} \to \mathscr{D}$ is a functor between $\infty$-categories. The reader can equally well, however, view $F$ as a functor between model categories that preserves weak equivalences. In that setting, many of the basic constructions are described by Kuhn in \cite{kuhn:2007}.

We will make considerable use of the notions of (homotopy) limit and colimit inside an $\infty$-category. We will refer to these simply as limits and colimits, though everywhere in this paper the appropriately homotopy-invariant concepts are intended. When working with a functor $F: \mathscr{C} \to \mathscr{D}$, we will usually require that $\mathscr{C}$ and $\mathscr{D}$ admit limits and colimits of particular shapes and that (especially in $\mathscr{D}$) certain limits and colimits commute. The relevant conditions will be made explicit when necessary.

\subsection*{Polynomial functors in the homotopy calculus}

To some extent, the theory of the calculus of functors is completely determined by making a choice as to which functors $F: \mathscr{C} \to \mathscr{D}$ are to be considered the analogues of degree $n$ polynomials. In Goodwillie's version, this choice is described in terms of cubical diagrams.

\begin{definition} \label{def:cubical}
An \emph{$n$-cube}\index{cubical diagram} in an $\infty$-category $\mathscr{C}$ is a functor $\mathscr{X}: \mathcal{P}(I) \to \mathscr{C}$, where $\mathcal{P}(I)$ is the poset of subsets of some finite set $I$ of cardinality $n$. An $n$-cube $\mathscr{X}$ is \emph{cartesian}\index{cubical diagram!cartesian} if the canonical map
\[ \mathscr{X}(\emptyset) \to \operatorname{holim}_{\emptyset \neq S \subseteq I} \mathscr{X}(S) \]
is an equivalence, and \emph{cocartesian}\index{cubical diagram!cocartesian} if
\[ \operatorname{hocolim}_{S \subsetneq I} \mathscr{X}(S) \to \mathscr{X}(I) \]
is an equivalence. When $n = 2$, these conditions reduce to the familiar notions of pullback and pushout, respectively. We also say that an $n$-cube $\mathscr{X}$ is \emph{strongly cocartesian}\index{cubical diagram!strongly cocartesian} if every $2$-dimensional face is a pushout. Note that a strongly cocartesian $n$-cube is also cocartesian if $n \geq 2$.
\end{definition}

We can now give Goodwillie's condition on a functor that plays the role of ``polynomial of degree $\leq n$''.

\begin{definition} \label{def:exc}
Let $\mathscr{C}$ be an $\infty$-category that admits pushouts. A functor $F: \mathscr{C} \to \mathscr{D}$ is \emph{$n$-excisive}\index{excisive functor} if it takes every strongly cocartesian $(n+1)$-cube in $\mathscr{C}$ to a cartesian $(n+1)$-cube in $\mathscr{D}$. We will say that $F$ is \emph{polynomial} if it is $n$-excisive for some integer $n$.

Let $\operatorname{Fun}(\mathscr{C},\mathscr{D})$ be the $\infty$-category of functors from $\mathscr{C}$ to $\mathscr{D}$, and let $\operatorname{Exc}_n(\mathscr{C},\mathscr{D})$ denote the full subcategory whose objects are the $n$-excisive functors.
\end{definition}

\begin{example}
In the somewhat degenerate case $n = 0$, Definition~\ref{def:exc} reduces to the statement that $F$ is $0$-excisive if and only if it is homotopically constant, i.e. $F$ takes every morphism in $\mathscr{C}$ to an equivalence in $\mathscr{D}$. (In an even more degenerate case, $F$ is $(-1)$-excisive if and only if $F(X)$ is a terminal object of $\mathscr{D}$ for all $X$ in $\mathscr{C}$.)
\end{example}

\begin{example} \label{ex:1-exc}
A functor $F: \mathscr{C} \to \mathscr{D}$ is $1$-excisive if and only if it takes pushout squares in $\mathscr{C}$ to pullback squares in $\mathscr{D}$. The prototypical example of such a functor (when $\mathscr{C}$ and $\mathscr{D}$ are both the category $\mathscr{T}op_*$ of based topological spaces) is
\[ X \mapsto \Omega^\infty(E \wedge X) \]
where $E$ is some spectrum. In fact, these examples constitute a classification of those functors that are $1$-excisive, \emph{reduced} (i.e. preserve the null object) and \emph{finitary} (i.e. preserve filtered colimits). This fact illustrates the key role played by stable homotopy theory in Goodwillie calculus.
\end{example}

\begin{remark}
It is notable that the identity functor $I: \mathscr{C} \to \mathscr{C}$ is typically \emph{not} $1$-excisive (or $n$-excisive for any $n$) unless $\mathscr{C}$ is a stable $\infty$-category. One might naturally think it would make more sense if a $1$-excisive functor were defined to preserve either pushouts or pullbacks, rather than to mix the two notions. Indeed, one can make such a definition and explore its properties. However, the notion defined by Goodwillie has turned out to be much more useful. This is partly because of the close connection to stable homotopy theory hinted at in Example~\ref{ex:1-exc}, but also because the fact that the identity functor is not $1$-excisive makes the theory \emph{more} useful rather than less, since, as we shall see, the Taylor tower of the identity functor provides an interesting decomposition of a space that we would not have if the identity were automatically linear.
\end{remark}

\begin{remark}
The property of a functor $F$ being $n$-excisive can also be described as a condition on sequences of $n+1$ morphisms in $\mathscr{C}$ with a common source, say $f_i: A \to X_i$ for $i = 0,\dots,n$. The condition relates the value of $F$ on the pushouts, over $A$, of all possible subsets of the sequence $(f_0,\dots,f_n)$. This can be viewed as the analogue of a way to specify when a function $f: \mathbb{R} \to \mathbb{R}$ is a degree $\leq n$ polynomial by considering the values of $f$ on sums of subsets of a set $(x_0,\dots,x_n)$ of real numbers.
\end{remark}

\begin{remark}
Johnson and McCarthy~\cite{johnson/mccarthy:2004} have given a different, slightly broader, definition for a functor $F: \mathscr{C} \to \mathscr{D}$ to be \emph{degree $\leq n$}: that the $(n+1)$-th cross-effect of $F$ vanishes. This choice leads to a different version of the Taylor tower described in the next section, although the difference seems not to be important in most cases of interest. In particular, for `analytic' functors, the two towers agree within the `radius of convergence'.
\end{remark}

As one might expect, the conditions of being $n$-excisive for varying $n$ are nested.

\begin{lemma} \label{lem:exc}
If $F: \mathscr{C} \to \mathscr{D}$ is $n$-excisive, then $F$ is also $n+1$-excisive. We therefore have a sequence of inclusions of subcategories
\[ \operatorname{Exc}_0(\mathscr{C},\mathscr{D}) \subseteq \operatorname{Exc}_1(\mathscr{C},\mathscr{D}) \subseteq \dots \subseteq \operatorname{Exc}_n(\mathscr{C},\mathscr{D}) \subseteq \operatorname{Exc}_{n+1}(\mathscr{C},\mathscr{D}) \subseteq \dots \]
\end{lemma}

\subsection*{Polynomial approximation in the homotopy calculus}

The fundamental construction in ordinary differential calculus is that of polynomial approximation: given a smooth function $f: \mathbb{R} \to \mathbb{R}$ and a real number $x$, there is a unique ``best'' degree $\leq n$ polynomial that approximates $f$ in a `neighbourhood' of $x$. To transfer this idea to the calculus of functors, we need to be able to compare the values of functors on objects in $\mathscr{C}$ that are related in some sense. In particular, we require a map between the objects in order to make this comparison. Thus the $n$-excisive approximation to a functor $F: \mathscr{C} \to \mathscr{D}$ in a neighbourhood of an object $X \in \mathscr{C}$ is only defined on objects $Y$ that come equipped with a map $Y \to X$ in $\mathscr{C}$, that is, on $Y$ in the slice $\infty$-category $\mathscr{C}_{/X}$.

\begin{definition}
We say that functors $\mathscr{C} \to \mathscr{D}$ \emph{admit $n$-excisive approximations}\index{excisive approximations} at $X$ in $\mathscr{C}$ if the composite
\[ \operatorname{Exc}_n(\mathscr{C}_{/X},\mathscr{D}) \hookrightarrow \operatorname{Fun}(\mathscr{C}_{/X},\mathscr{D}) \to \operatorname{Fun}(\mathscr{C},\mathscr{D}) \]
has a left adjoint, which, when it exists, we write $P^X_n$. Here (as everywhere in this article) we mean an adjunction in the $(\infty,1)$-categorical sense: see, for example Lurie~\cite[5.2.2.1]{lurie:2009} or Riehl-Verity~\cite[1.1]{riehl/verity:2013}.
\end{definition}

\begin{remark}
Biedermann, Chorny and R\"{o}ndigs showed in \cite{biedermann/chorny/rondigs:2007} that the $n$-excisive approximation is a left Bousfield localization of a suitable category of functors, and it provides a best approximation \emph{on the right} to a given functor $F: \mathscr{C} \to \mathscr{D}$ by one that is $n$-excisive. Explicitly, the $n$-excisive approximation to $F$ at $X$, if it exists, consists of a natural transformation
\[ F \to P^X_nF \]
of functors $\mathscr{C}_{/X} \to \mathscr{D}$ that is initial (up to homotopy) among natural transformations from $F$ (restricted to $\mathscr{C}_{/X}$) to an $n$-excisive functor.
\end{remark}

The first main theorem of Goodwillie calculus is that such $n$-excisive approximations exist under mild conditions on $\mathscr{C}$ and $\mathscr{D}$. The following result is stated by Lurie, but the proof is no different than that given originally by Goodwillie in the context of functors of topological spaces and spectra.

\begin{theorem}[{Goodwillie~\cite[1.13]{goodwillie:2003}}, {Lurie~\cite[6.1.1.10]{lurie:2017}}] \label{thm:approx}
Let $\mathscr{C}$ and $\mathscr{D}$ be $\infty$-categories, and suppose that $\mathscr{C}$ has pushouts, and that $\mathscr{D}$ has sequential colimits, and finite limits, which commute. Then functors $\mathscr{C} \to \mathscr{D}$ admit $n$-excisive approximations at any object $X \in \mathscr{C}$.
\end{theorem}

\begin{example}
The $0$-excisive approximation to $F$ at $X$ is, as you would expect, equivalent to the constant functor with value $F(X)$.
\end{example}

There is no loss of generality (and the notation is simpler) if we focus on the case where $X$ is a terminal object of $\mathscr{C}$. In this setting, the $n$-excisive approximation to $F: \mathscr{C} \to \mathscr{D}$ is another functor from $\mathscr{C}$ to $\mathscr{D}$, which we simply denote by $P_nF$.

\begin{example}
Goodwillie gives an explicit construction of $P_nF$ which is easiest to describe when $n = 1$ and where $\mathscr{C}$ and $\mathscr{D}$ are both pointed with $F: \mathscr{C} \to \mathscr{D}$ \emph{reduced}, i.e. preserving the null object. In this case, $P_1F(Y)$ can be written as the colimit of the following sequence of maps
\[ F(Y) \to \Omega F(\Sigma Y) \to \Omega^2 F(\Sigma^2 Y) \to \dots \]
where $\Sigma$ and $\Omega$ are the suspension and loop-space functors for $\mathscr{C}$ and $\mathscr{D}$ respectively. For $F: \mathscr{T}op_* \to \mathscr{T}op_*$, we have
\[ P_1F(Y) \simeq \Omega^\infty \bar{F} (\Sigma^\infty Y) \]
where $\bar{F}$ is a reduced $1$-excisive functor from spectra to spectra. For $Y$ a finite CW-complex, such a functor can be written in the form of Example~\ref{ex:1-exc}:
\[ P_1F(Y) \simeq \Omega^\infty(\partial_1F \wedge Y) \]
where $\partial_1F$ is a spectrum which we refer to as the \emph{(first) derivative of $F$} (at the one-point space $*$).
\end{example}

\begin{example}
The $1$-excisive approximation to the identity functor on \emph{based} spaces is simply the stable homotopy functor
\[ P_1I(Y) \simeq \Omega^\infty \Sigma^\infty Y =: Q(Y) \]
or equivalently, $\partial_1I \simeq S^0$, the sphere spectrum.

The unbased case is slightly more subtle: the $1$-excisive approximation to the identity functor in that context can be written as
\[ P_1I(Y) \simeq \operatorname{hofib}(Q(Y_+) \to Q(S^0)) =: \tilde{Q}(Y) \]
where the homotopy fibre is calculated over the point in $Q(S^0)$ corresponding to the \emph{identity} map on the sphere spectrum (as opposed to the null map).
\end{example}

\subsection*{Taylor tower and convergence}

The explicit description of $P_nF$ is hard to make use of for $n > 1$, even in the case of the identity functor. The real power of the calculus of functors derives from the tower formed by the $n$-excisive approximations for varying $n$.

\begin{definition}
The \emph{Taylor tower}\index{Taylor tower!of a functor} (or \emph{Goodwillie tower}) of $F: \mathscr{C} \to \mathscr{D}$ at $X \in \mathscr{C}$ is the sequence of natural transformations (of functors $\mathscr{C}_{/X} \to \mathscr{D}$):
\[ F \to \dots \to P^X_{n+1}F \to P^X_nF \to \dots \to P^X_1F \to P^X_0F \simeq F(X) \]
where it follows from the universal property of $P^X_{n+1}F$, and Lemma~\ref{lem:exc}, that each $F \to P^X_nF$ factors as shown.
\end{definition}

For a given map $f: Y \to X$ in $\mathscr{C}$, the Taylor tower provides a sequence of factorizations of $F(f): F(Y) \to F(X)$. Ideally, we would be able to recover the value $F(Y)$ from this sequence of approximations $P^X_nF(Y)$ in the following way.

\begin{definition}
The Taylor tower of $F: \mathscr{C} \to \mathscr{D}$ \emph{converges}\index{convergence of Taylor tower} at $Y \in \mathscr{C}_{/X}$ if the induced map
\[ F(Y) \to \operatorname{holim}_n P^X_nF(Y) \]
is an equivalence in $\mathscr{D}$.
\end{definition}

Arguably the question of convergence of the Taylor tower is the most important step in actually applying the calculus of functors to a particular functor $F$. Very general approaches to proving convergence seem rare, but Goodwillie has developed an extensive set of tools, based on connectivity estimates, in the contexts of topological spaces and spectra.

These tools are based on measuring the failure of a functor $F$ to be $n$-excisive via connectivity. This can be done by applying $F$ to a strongly cocartesian $(n+1)$-cube, and examining the failure of the resulting cube to be cartesian, in terms of the connectivity of the map from the initial vertex to the homotopy limit of the rest of the diagram.

Roughly speaking, a functor $F$ is \emph{stably $n$-excisive}\index{stably excisive} if this connectivity is controlled relative to the connectivities of the maps in the original cocartesian cube. Goodwillie then says that $F$ is \emph{$\rho$-analytic}\index{analytic functor}, for some real number $\rho$, if it is stably $n$-excisive for all $n$ where these connectivity estimates depend linearly on $n$ with slope $\rho$. See \cite{goodwillie:1991} for complete details.

The upshot of these definitions is the following theorem.

\begin{theorem}[Goodwillie~{\cite[1.13]{goodwillie:2003}}]
Let $F: \mathscr{C} \to \mathscr{D}$ be a $\rho$-analytic functor where $\mathscr{C}$ and $\mathscr{D}$ are each either spaces or spectra. Then the Taylor tower of $F$ at $X \in \mathscr{C}$ converges on those objects $Y$ in $\mathscr{C}_{/X}$ whose underlying map $Y \to X$ is $\rho$-connected.
\end{theorem}

\begin{examples}
The identity functor $I: \mathscr{T}op \to \mathscr{T}op$ is $1$-analytic \cite[4.3]{goodwillie:1991}. This depends on higher dimensional versions of the Blakers-Massey Theorem relating pushouts and pullbacks in $\mathscr{T}op$. Indeed the usual Blakers-Massey Theorem implies that $I$ is stably $1$-excisive. Waldhausen's algebraic $K$-theory of spaces functor $A: \mathscr{T}op \to {\mathscr{S}p}$ is also $1$-analytic \cite[4.6]{goodwillie:1991}. In particular, the Taylor towers at $*$ of both of these functors converge on simply-connected spaces.
\end{examples}

\begin{definition} \label{def:gss}
Let $F: \mathscr{C} \to \mathscr{T}op_*$ be a homotopy functor with values in pointed spaces. The \emph{Goodwillie spectral sequence} associated to $F$ at $Y \in \mathscr{C}_{/X}$ is the homotopy spectral sequence of the tower of pointed spaces $(P^X_nF(Y))_{n \geq 1}$~\cite[section IX.4]{bousfield/kan:1972}. This spectral sequence converges to
\[ \pi_*P^X_{\infty}F(Y) \]
where $P^X_{\infty}F := \operatorname{holim}_n P^X_nF$, and has $E^1$-term given by the homotopy groups of the \emph{layers} of the Taylor tower, i.e.
\[ E^1_{s,t} \cong \pi_{t-s}D^X_sF(Y) \]
where $D^X_sF := \operatorname{hofib}(P^X_sF \to P^X_{s-1}F)$. If $F$ is analytic (and $Y \to X$ is suitably connected), then this spectral sequence converges strongly~\cite[section IX.5]{bousfield/kan:1972}.
\end{definition}

\begin{example} \label{ex:gss-id}
For the identity functor on based spaces, the Goodwillie spectral sequence at $Y \in \mathscr{T}op_*$ takes the form
\[ E^1_{s,t} = \pi_{t-s}D_sI(Y) \]
and converges to the homotopy groups of $Y$ when $Y$ is simply-connected (or, more generally, when the Taylor tower of the identity converges for $Y$. This spectral sequence has been studied extensively by Behrens~\cite{behrens:2012}. The spectral sequence also motivates the study of the layers $D^X_nF$ of the Taylor tower in general, and we turn to these now.
\end{example}

\section{The Classification of Homogeneous Functors}

Let $F: \mathscr{C} \to \mathscr{D}$ be a homotopy functor, where $\mathscr{C}$ and $\mathscr{D}$ are as in Theorem~\ref{thm:approx}, and suppose further that $\mathscr{D}$ is a \emph{pointed} $\infty$-category. Then we make the following definition, generalizing that given at the end of the previous section.

\begin{definition}
The $n$-th \emph{layer}\index{layer of the Taylor tower} of the Taylor tower of $F$ at $X$ is the functor $D^X_nF: \mathscr{C}_{/X} \to \mathscr{D}$ given by
\[ D^X_nF(Y) := \operatorname{hofib}(P^X_nF(Y) \to P^X_{n-1}F(Y)). \]
\end{definition}

These layers play the role of \emph{homogeneous} polynomials in the theory of calculus, and satisfy the following definition.

\begin{definition}
Let $F: \mathscr{C} \to \mathscr{D}$ be a homotopy functor that admits $n$-excisive approximations, and where $\mathscr{D}$ has a terminal object $*$. We say that $F$ is \emph{$n$-homogeneous}\index{homogeneous functor} if $F$ is $n$-excisive and $P_{n-1}F \simeq *$.
\end{definition}

Another of Goodwillie's main theorems from \cite{goodwillie:2003} is the existence of a natural delooping, and consequently a classification, of homogeneous functors. To state this in the generality of $\infty$-categories, we first recall that any suitable pointed $\infty$-category $\mathscr{C}$ admits a \emph{stabilization}, that is a stable $\infty$-category ${\mathscr{S}p}(\mathscr{C})$ together with an adjunction
\[ \Sigma^\infty_{\mathscr{C}} : \mathscr{C} \rightleftarrows {\mathscr{S}p}(\mathscr{C}) : \Omega^\infty_{\mathscr{C}} \]
that generalizes the suspension spectrum / infinite-loop space adjunction, which we write simply as $(\Sigma^\infty,\Omega^\infty)$, for $\mathscr{C} = \mathscr{T}op_*$.

\begin{theorem} \label{thm:hom}
Let $F: \mathscr{C} \to \mathscr{D}$ be an $n$-homogeneous functor between pointed $\infty$-categories. Then there is a \emph{symmetric multilinear} functor $H: {\mathscr{S}p}(\mathscr{C})^n \to {\mathscr{S}p}(\mathscr{D})$, and a natural equivalence
\[ F(X) \simeq \Omega^\infty_{\mathscr{D}} \left[H(\Sigma^\infty_{\mathscr{C}}X, \dots, \Sigma^\infty_{\mathscr{C}}X)_{h\Sigma_n}\right] \]
where we are taking the homotopy orbit construction with respect to the action of the symmetric group $\Sigma_n$ that permutes the entries of $H$.
\end{theorem}

\begin{example}\label{example:hom}
For functors $\mathscr{T}op_* \to \mathscr{T}op_*$, a symmetric multilinear functor is uniquely determined (on finite CW-complexes at least) by a single spectrum with a symmetric group action. Applying this classification to the layers of the Taylor tower of $F: \mathscr{T}op_* \to \mathscr{T}op_*$, we get an equivalence
\[ D_nF(X) \simeq \Omega^\infty(\partial_nF \wedge (\Sigma^\infty X)^{\wedge n})_{h\Sigma_n} \]
where $\partial_nF$ is a spectrum with a (naive) action of the symmetric group $\Sigma_n$, which we refer to as the \emph{$n$-th derivative}\index{derivative of a functor} of $F$ (at $*$). Similar formulas apply when $\mathscr{C}$ and/or $\mathscr{D}$ is ${\mathscr{S}p}$ instead of $\mathscr{T}op_*$, and sense can be made of the object $\partial_nF$ for more general $\mathscr{C}$ and $\mathscr{D}$, though in such cases $\partial_nF$ is a diagram of spectra (indexed by generators for the stable $\infty$-categories ${\mathscr{S}p}(\mathscr{C})$ and ${\mathscr{S}p}(\mathscr{D})$) rather than a single spectrum, e.g. see~\cite[1.1]{ching:2018}.
\end{example}

\begin{example} \label{ex:id}
The $n$-th derivative of the identity functor $I: \mathscr{T}op_* \to \mathscr{T}op_*$\index{derivative of the identity functor} was calculated by Brenda Johnson in \cite{johnson:1995}. The spectrum $\partial_nI$ is equivalent to a wedge of $(n-1)!$ copies of the $(1-n)$-sphere spectrum. In particular, $\partial_1I \simeq S^0$ (as mentioned above) and $\partial_2I \simeq S^{-1}$ (with trivial $\Sigma_2$-action), so
\[ D_2I(X) \simeq \Omega^{\infty}\Sigma^{-1}(\Sigma^\infty X)^{\wedge 2}_{h\Sigma_2}. \]
We can now attempt to calculate $P_2I$ using the fibre sequence $D_2I \to P_2I \to P_1I$. This takes the form
\[ \Omega^{\infty}\Sigma^{-1}(\Sigma^\infty X)^{\wedge 2}_{h\Sigma_2} \to P_2I(X) \to \Omega^\infty \Sigma^\infty X. \]
Goodwillie's results imply that this sequence deloops, so that $P_2I(X)$ can be written as the fibre of a certain natural transformation.
\end{example}

\begin{example} \label{ex:SO}
The $n$-th derivative of the functor $\Sigma^\infty \Omega^\infty: {\mathscr{S}p} \to {\mathscr{S}p}$ is equivalent to $S^0$ (with trivial $\Sigma_n$-action)~\cite[Corollary 1.3]{ahearn/kuhn:2002}. Therefore
\[ D_n(\Sigma^\infty \Omega^\infty)(X) \simeq X^{\wedge n}_{h\Sigma_n}. \]
This tells us that, for a spectrum $X$, the spectrum $\Sigma^\infty \Omega^\infty X$ is given by piecing together the extended powers $X^{\wedge n}_{h\Sigma_n}$. When $X$ is a $0$-connected suspension spectrum, Snaith splitting provides an equivalence
\[ \Sigma^\infty \Omega^\infty X \simeq \bigvee_{n \geq 1} X^{\wedge n}_{h\Sigma_n} \]
which can be interpreted as the splitting of the Taylor tower. For arbitrary $X$, the layers of the tower are pieced together in a less trivial way.
\end{example}

\subsection*{Splitting results for the Taylor tower} \label{sec:splitting}

As illustrated in Example~\ref{ex:SO}, the simplest situation holds when the Taylor tower is a product of its layers:
\[ P_nF(X) \simeq \prod_{k = 0}^{n} D_kF(X). \]

\begin{example}
Let $K$ be a $d$-dimensional based finite CW-complex and consider the functor $\Sigma^\infty \operatorname{Hom}_*(K,-): \mathscr{T}op_* \to {\mathscr{S}p}$ where $\operatorname{Hom}_*(-,-)$ denotes the space of basepoint-preserving maps between two based spaces. The first author, in~\cite{arone:1999}, proved (1) that
\[ \partial_n(\Sigma^\infty \operatorname{Hom}_*(K,-)) \simeq  \operatorname{Map}(\Sigma^\infty K^{\wedge n}/\Delta^nK,S^0) \]
where $\operatorname{Map}(-,-)$ is the mapping spectrum construction, and $\Delta^nK$ denotes the `fat diagonal', i.e. the subspace of $K^{\wedge n}$ consisting of those points $(k_1,\dots,k_n)$ where $k_i = k_j$ for some $i \neq j$, (2) that the Taylor tower converges on $d$-connected $X$, and (3) when $K$ is a parallelizable $d$-dimensional manifold and $X$ is a $d$-fold suspension, that the Taylor tower splits. Thus in the latter case we have:
\[ \Sigma^\infty \operatorname{Hom}_*(K,X) \simeq \prod_{n = 1}^{\infty} \operatorname{Map}(\Sigma^\infty K^{\wedge n}/\Delta^nK,\Sigma^\infty X^{\wedge n})_{h\Sigma_n}. \]
\end{example}

Kuhn has proved, in \cite{kuhn:2004}, the following splitting result which reveals some of the interesting interaction between Goodwillie calculus, Tate cohomology and chromatic homotopy theory.

\begin{theorem}[Kuhn] \label{thm:Kuhn}
Let $F: {\mathscr{S}p} \to {\mathscr{S}p}$ be a homotopy functor. Then the Taylor tower of $F$ splits after $T(k)$-localization. (Here $T(k)$ denotes the spectrum given by the telescope of a $v_k$-self map of a finite type-$k$ complex.) In other words
\[ L_{T(k)}P_nF(X) \simeq \prod_{j = 1}^{n} L_{T(k)}D_jF(X). \]
\end{theorem}

The proof of Theorem~\ref{thm:Kuhn} relies on a number of interesting ingredients, in particular the vanishing of the $T(k)$-localization of the Tate construction associated to a finite group action on a spectrum. The specific part coming from functor calculus, however, is the following result of McCarthy~\cite{mccarthy:2001}.

\begin{theorem}[McCarthy] \label{thm:McCarthy}
For a functor $F: {\mathscr{S}p} \to {\mathscr{S}p}$ that preserves filtered colimits, there is a natural homotopy pullback square of the form
\[ \begin{diagram}
  \node{P_nF(X)} \arrow{s} \arrow{e} \node{(\partial_nF \wedge X^{\wedge n})^{h\Sigma_n}} \arrow{s} \\
  \node{P_{n-1}F(X)} \arrow{e} \node{(\partial_nF \wedge X^{\wedge n})^{t\Sigma_n}}
\end{diagram} \]
Here $Y^{tG}$ denotes the Tate construction of the action of a finite group $G$ on a spectrum $Y$, that is, the cofibre of the norm map $N: Y_{hG} \to Y^{hG}$, and the right-hand vertical map above is the canonical map from the homotopy fixed points to the Tate construction.
\end{theorem}

Note that the induced map between the homotopy fibres of the vertical maps in this diagram is the equivalence $D_nF(X) \; \tilde{\longrightarrow} \; (\partial_nF \wedge X^{\wedge n})_{h\Sigma_n}$ that appears in the classification of $n$-homogeneous functors from spectra to spectra.

Kuhn's theorem, roughly speaking, follows from McCarthy's by taking $T(k)$-localization of the homotopy pullback square, and using the vanishing of the Tate construction.

The results quoted here begin to address the question of how the layers of the Taylor tower of a homotopy functor are pieced together to form the tower itself, and they start to illustrate the role of the Tate spectrum construction in that picture. We will return to this topic in Section~\ref{sec:operads}.

\section{The Taylor Tower of the Identity Functor for Based Spaces} \label{sec:id}

Let $I$ be the identity functor from the category of based spaces to itself. From the perspective of functor calculus, this is a highly non-trivial object. As noted above, $I$ is not an $n$-excisive functor for any $n$, and its derivatives have a lot of structure. Applying homotopy groups to the Taylor tower, we get a sequence
\[ \pi_*X \to \dots \to \pi_*P_nI(X) \to \pi_*P_{n-1}I(X) \to \dots \to \pi_*P_1(X) = \pi^s_*X \]
interpolating between the unstable and stable homotopy groups of a space $X$.

The first step to understanding the Taylor tower of a functor is to calculate the derivatives (and hence the layers). For the identity functor, the derivatives were, as mentioned above, calculated by Johnson in \cite{johnson:1995}. Here we give the reformulation of her result produced in \cite{arone/mahowald:1999}.

\begin{definition}\label{def:tn}
Let $\mathcal P_n$ be the poset of partitions of the set $\{1, \ldots, n\}$, ordered by refinement. Let $|\mathcal P_n|$ be the geometric realization of $\mathcal P_n$. Note that $\mathcal P_n$ has both an initial and a final object. It follows in particular that $|\mathcal P_n|$ is contractible. Let $\partial |\mathcal P_n|$ be the subcomplex of $|\mathcal P_n|$ spanned by simplices that do not contain both the initial and final element as vertices. Let $T_n=|\mathcal P_n|/\partial |\mathcal P_n|$.
\end{definition}

It is easy to see that $T_n$ is an $n-1$-dimensional complex with an action of the symmetric group $\Sigma_n$. It is well known that non-equivariantly $T_n$ is equivalent to $\bigvee_{(n-1)!} S^{n-1}$. In fact, this equivalence holds already after restricting the action from $\Sigma_n$ to $\Sigma_{n-1}$, where $\Sigma_{n-1}$ is considered a subgroup of $\Sigma_n$ in the standard way. This means that there is a $\Sigma_{n-1}$-equivariant equivalence $T_n\simeq {\Sigma_{n-1}}_+\wedge S^{n-1}$ \cite{arone/brantner:2018}.

\begin{theorem}\cite{johnson:1995, arone/mahowald:1999}  \label{theorem: brenda}
There is a $\Sigma_n$-equivariant equivalence of spectra\index{derivative of the identity functor}
\[
\partial_nI \simeq \mathbb{D}(T_n)
\]
between the $n$-th derivative of the identity functor and the Spanier-Whitehead dual of the complex $T_n$.
\end{theorem}

It turns out that the Taylor tower of the identity functor has some rather special properties when evaluated at a sphere. The results are cleanest to state for an odd-dimensional sphere, so we will mostly focus on this case.

We begin by noting that rationally the tower is constant.
\begin{theorem} \cite[Proposition 3.1]{arone/mahowald:1999}\label{thereom: rational}
Let $X$ be an odd-dimensional sphere. The spectrum
\[(\partial_nI \wedge X^{\wedge n})_{h\Sigma_n}\]
is rationally contractible for $n>1$.
\end{theorem}
\begin{proof}[Proof sketch] Since $\partial_nI \simeq \mathbb{D}(T_n)$, it is enough to prove that $\Sigma^{\infty} X^{\wedge n}_{hG}$ has trivial rational homology when $G$ is any isotropy group of $T_n$. This follows by an easy spectral sequence argument.
\end{proof}
This strengthens the following classical computation of Serre:
\begin{corollary}[Serre, 1953] \label{corollary: serre}
When $X$ is an odd-dimensional sphere, the map $X\to \Omega^\infty\Sigma^\infty X$ is a rational homotopy equivalence.
\end{corollary}
It follows from the theorem that for $X$ an odd sphere, the homology of $(\partial_nI \wedge X^{\wedge n})_{h\Sigma_n}$ is torsion for all $n>1$. The following theorem gives considerably more information about how the torsion is distributed among the layers.
\begin{theorem}\cite{arone/mahowald:1999, arone/dwyer:2001}\label{theorem: primary}
Let $X$ be an odd-dimensional sphere, and let $p$ be a prime. The homology with mod $p$ coefficients of the spectrum $(\partial_nI \wedge X^{\wedge n})_{h\Sigma_n}$ is non-trivial only if $n$ is a power of $p$.
\end{theorem}
Note that it follows that if $n$ is not a prime power then the spectrum $(\partial_nI \wedge X^{\wedge n})_{h\Sigma_n}$ is contractible, since it is a connective spectrum of finite type whose homology is trivial with rational and mod $p$ coefficients for all primes $p$. Theorem~\ref{theorem: primary} implies that if one is willing to pick a prime $p$ and localize all spaces at $p$, then the only non-trivial layers in the tower are the ones numbered by powers of $p$. Thus the tower converges exponentially faster than usual in this case.

Theorem~\ref{theorem: primary} was first proved in~\cite{arone/mahowald:1999}, by a brute force calculation of the homology with mod $p$ coefficients. A more conceptual proof was given in~\cite{arone/dwyer:2001}. Furthermore, when $n=p^k$, it turns out that the spectrum $(\partial_nI \wedge X^{\wedge n})_{h\Sigma_n}$ is closely related to well-studied spectra in the literature (remember that $X$ is an odd sphere throughout the discussion). In particular, this spectrum is equivalent to (the $k$-fold desuspension of) the direct summand of the spectrum $\Sigma^\infty X^{\wedge p^k}_{h({\mathbb Z}/p)^k}$ split off by the Steinberg idempotent. Such wedge summands were studied by Mitchell, Kuhn, Priddy and others. In the case $X=S^1$, this spectrum is equivalent (up to a suspension) to the $n$-th subquotient in the filtration of $H\mathbb Z$ by symmetric powers of the sphere spectrum.

It turns out that Serre's theorem on rational homotopy groups of spheres (corollary~\ref{corollary: serre}) admits a rather dramatic generalization to $v_k$-periodic homotopy for all $k$ (rational homotopy fits in as the case $k=0$). We will now state the result, and then spend the rest of the section explaining what it says and outlining its proof.
\begin{theorem}\cite{arone/mahowald:1999} \label{thm:vk-periodic}
Let $X$ be an odd-dimensional sphere, and work $p$-locally for a prime $p$. For $k\ge 0$, the map $X\mapsto P_{p^k}I(X)$ is a $v_k$-periodic equivalence.
\end{theorem}
We will now take a detour to review some background on $v_k$-periodic homotopy. For an omnibus reference on this material we suggest Ravenel's orange book~\cite{ravenel:1992} (note that a revised version is available online). Fix a prime $p$ and let all spaces be implicitly localized at $p$. Homology and cohomology groups will be taken with mod $p$ coefficients. Recall that for each integer $n\ge 0$ there is a generalized homology theory $K(n)$ called the $n$-th Morava $K$-theory. For $n=0$, $K(0)=H{\mathbb Q}$, and for $n>0$ the coefficient ring of $K(n)$ is $K(n)_*={\mathbb F}_p[v_n, v_n^{-1}]$, where $|v_n|=2p^n-2$.
\begin{definition}
A finite complex $V$ is said to be of \emph{type $k$} if $K(n)_*V$ is trivial for $n< k$ and non-trivial for $n=k$.
\end{definition}
By the Periodicity Theorem~\cite{devinatz/hopkins/smith:1988},~\cite[Theorem 1.5.4]{ravenel:1992}, for each $k\ge 0$ there exists a finite complex of type $k$. Furthermore, suppose $k\ge 1$, and $V_k$ is a complex of type $k$. Then there exists a self-map $f\colon \Sigma^{d|v_k|+i} V_k \to \Sigma^iV_k$ for some $i\ge 0, d \ge 1$, whose effect on $K(n)_*$ is an isomorphism for $n=k$ and zero for $n>k$. 
A map with these properties is called a $v_k$-periodic map. By the uniqueness part of the Periodicity Theorem, any two $v_k$-periodic self maps of $V_k$ are equivalent after taking some suspensions and iterations.

Suppose $X$ is either a pointed space or a spectrum, and $V$ is a finite complex with a basepoint. If $X$ is a space, let $X^V$ denote the space of basepoint-preserving maps from $V$ to $X$. If $X$ is a spectrum, then let $X^V$ denote the mapping spectrum from $\Sigma^\infty V$ to $X$. Clearly, if $X$ is a spectrum then $\Omega^\infty (X^V)\cong (\Omega^\infty X)^V$. Let $V_k$ be a complex of type $k$. By replacing $V_k$ with some suspension thereof if necessary, we may assume that $V_k$ has a self map of the form $\Sigma^{d|v_k|} V_k \to V_k$. This map induces a map $X^{V_k}\to \Omega^{d|v_k|} X^{V_k}$, where $X$ is still either a pointed space or a spectrum. We can form a mapping telescope as follows
\[ v_k^{-1}X^{V_k}:=\operatorname{hocolim} (X^{V_k}\to \Omega^{d|v_k|} X^{V_k} \to  \Omega^{2d|v_k|} X^{V_k} \to \cdots). \]
The mapping telescope serves as the definition of $v_k^{-1}X^{V_k}$. If $X$ is a space, then $v_k^{-1}X^{V_k}$ is, by definition, a space, that is easily seen to be an infinite loop space. If $X$ is a spectrum then $v_k^{-1}X^{V_k}$ is a spectrum. Clearly, if $X$ is a spectrum then
\[ \Omega^\infty(v_k^{-1}X^{V_k})\simeq v_k^{-1}(\Omega^\infty X)^{V_k}. \]

We define the $v_k$-periodic homotopy groups of $X$ with coefficients in $V_k$ as follows.
\[ v_k^{-1}\pi_*(X; V_k):= \pi_*\left(v_k^{-1}X^{V_k}\right)\cong \operatorname{colim} (\pi_*(X^{V_k})\to \pi_{*+d|v_k|}(X^{V_k})\to \cdots). \]
As before, the groups $v_k^{-1}\pi_*(X; V_k)$ are defined if $X$ is either a space or a spectrum. It is easy to see that if $X$ is a spectrum then there is a canonical isomorphism
\begin{equation}\label{equation: loopinfinity}
v_k^{-1}\pi_*(X; V_k)\cong v_k^{-1}\pi_*(\Omega^\infty X; V_k)
\end{equation}
The groups $v_k^{-1}\pi_*(X; V_k)$ are periodic with period that divides $d|v_k|$. They depend on the choice of $V_k$, but by the uniqueness statement above they do not depend on the self-map of $V_k$ (up to a dimension shift).

While the groups $v_k^{-1}\pi_*(X; V_k)$ depend on a choice of $V_k$, the following is a well-known consequence of the Thick Subcategory Theorem
\begin{proposition}\label{proposition: v_k independent}
Let $f\colon X\to Y$ be a map of spaces. If there exists one complex $V_k$ of type $k$ for which $f$ induces an isomorphism $v_k^{-1}\pi_*(X; V_k)\to v_k^{-1}\pi_*(Y; V_k)$ then $f$ induces an isomorphism for every such $V_k$.
\end{proposition}
We say that a map $X\to Y$ is a $v_k$-equivalence if it induces an isomorphism on $v_k$-periodic groups for some, and therefore every, choice of a complex $V_k$ of type $k$. This explains the use of the term in Theorem~\ref{thm:vk-periodic}.

A convenient subclass of complexes of type $k$ are ones that are {\it strongly of type $k$}. They are defined by certain freeness properties of the Steenrod algebra action, and the requirement that the AHSS for Morava $K$-theory collapses. For a precise definition see~\cite[Definition 6.2.3]{ravenel:1992}. Proposition~\ref{proposition: strong type k} summarizes the relevant facts about complexes that are strongly of type $k$. Recall that the Adams spectral sequence (ASS) has the following form, where $E$ and $F$ are spectra
\[ \operatorname{Ext}_{\mathcal A}^{s,t}(H^*(F), H^*(E))\Rightarrow \pi_{t-s}(\operatorname{Map}(E, F)^{\wedge}_p). \]

\begin{proposition}\cite[Sections 6.2--6.4]{ravenel:1992} \label{proposition: strong type k}
For every $k$ there exists a complex strongly of type $k$. Let $V_k$ be such a complex. After some suspension, $V_k$ has a $v_k$-periodic self map $\Sigma^{d|v_k|} V_k \to V_k$ whose stabilization is represented in the second page of the ASS by an element of $\operatorname{Ext}_{\mathcal A}^{d,d|v_k|+d}(H^*(V_k), H^*(V_k))$. In particular, the self-map has Adams filtration $d$.
\end{proposition}

It is often convenient to bigrade the Adams spectral sequence by $(t-s, s)$, so that the horizontal axis corresponds to the topological degree and the vertical axis is the Adams filtration. Vanishing lines in the spectral sequence are calculated in these coordinates. Thus a vanishing line gives an upper bound on the possible Adams filtration of a non-zero element of the homotopy groups of a spectrum in terms of its topological dimension. If one uses the grading $(t-s, s)$ then Proposition~\ref{proposition: strong type k} says the self map of $V_k$ is represented by multiplication by an element on a line of slope $\frac{1}{|v_k|}$ and intercept zero.

Let $\operatorname{End}(\Sigma^\infty V_k)\simeq \Sigma^\infty V_k\wedge D(\Sigma^\infty V_k)$ be the endomorphism spectrum of $\Sigma^\infty V_k$. The self map $\Sigma^{d|v_k|} V_k \to V_k$ gives rise to an element of $\pi_{d|v_k|}(\operatorname{End}(\Sigma^\infty V_k))$. Proposition~\ref{proposition: strong type k} says that this element has Adams filtration $d$. It follows that if $E$ is any spectrum, then the induced map $E^{V_k} \to \Omega^{d|v_k|}E^{V_k}$ also has Adams filtration at least $d$. In fact, a slightly stronger statement is true: the self map of $E^{V_k}$ is represented in ASS by an operation that raises topological degree by $d|v_k|$ and raises Adams filtration by $d$. In other words, it moves elements in the ASS for $E^{V_k}$ along a line of slope $\frac{1}{|v_k|}$. It follows that if the ASS for $E^{V_k}$ has a vanishing line of slope smaller than $\frac{1}{|v_k|}$ then the action of the self map of $V_k$ on the homotopy groups of $E^{V_k}$ is nilpotent.

It turns out that layers in the Goodwillie tower of the identity have good vanishing lines. So now it is time to get back to the Goodwillie tower.

We remind the reader that $X$ is an odd sphere and everything is localised at a prime $p$. By Theorem~\ref{theorem: primary}, the only non-trivial layers of the Goowillie tower of the identity at $X$ are the ones indexed by powers of $p$. Their underlying spectra have the form $(\partial_{p^k}I \wedge X^{\wedge p^k})_{h\Sigma_{p^k}}$. It turns out that these spectra have interesting properties in $v_k$-periodic homotopy. Our way to see it is via their cohomology with the Steenrod action. For an integer $k\ge 0$ let $A_k$ be the subalgebra of the Steenrod algebra generated by $\{Sq^1, Sq^2, Sq^4, \ldots, Sq^{2^k}\}$ for $p=2$ and by $\{\beta, P^1, P^p, \ldots, P^{p^{k-1}}\}$ for $p>2$.
\begin{proposition}\label{proposition: A_k free}
The cohomology of $(\partial_{p^k}I \wedge X^{\wedge p^k})_{h\Sigma_{p^k}}$ is free over $A_{k-1}$
\end{proposition}
Again, this proposition was first proved in~\cite{arone/mahowald:1999} by a brute force computation. The cohomology of the spectrum $(\partial_{p^k}I \wedge X^{\wedge p^k})_{h\Sigma_{p^k}}$ was calculated explicitly, and it was observed that in the case $X=S^1$ it is isomorphic (up to degree shift) to the cohomology of the quotient of symmetric product spectra
\[ \mathrm{Sp}^{p^k}(S^0)/\mathrm{Sp}^{p^k-1}(S^0). \]
The cohomology of the latter is $A_{k-1}$-free by a theorem of Welcher~\cite{welcher:1981}, and Welcher's argument can be adapted to the case of a more general $X$.

It was pointed out to us by Nick Kuhn that there is a more direct way to deduce proposition~\ref{proposition: A_k free} from Welcher's result. We know from~\cite{arone/dwyer:2001} that the spectrum $(\partial_{p^k}I \wedge S^{p^k})_{h\Sigma_{p^k}}$ is  homotopy equivalent (up to a suspension) to
\[ \mathrm{Sp}^{p^k}(S^0)/\mathrm{Sp}^{p^k-1}(S^0), \]
and for a general odd sphere $X$, the spectrum $(\partial_{p^k}I \wedge X^{\wedge p^k})_{h\Sigma_{p^k}}$ is (roughly speaking) a Thom spectrum over the case $X=S^1$. Thus there is a Thom isomorphism between the cohomologies of the two spectra:
\[ H^*\left(\left(\partial_{p^k}I \wedge S^{p^k}\right)_{h\Sigma_{p^k}}\right)\cong H^{*+2lp^k}\left(\left(\partial_{p^k}I \wedge S^{(2l+1)p^k}\right)_{h\Sigma_{p^k}}\right) \]
Furthermore, one can show that $A_{k-1}$ acts trivially on the Thom class. This can be done by identifying the Thom class with a power of the top Dickson invariant, and using the formulas in~\cite{wilkerson:1983}. From here it follows that in our case the Thom isomorphism respects the $A_{k-1}$-module structure.

\begin{remark}
It seems likely that Proposition~\ref{proposition: A_k free} can also be proved by adapting the methods of Steve Mitchell~\cite{mitchell:1985}.
\end{remark}
Proposition~\ref{proposition: A_k free} has consequences regarding the $v_k$-periodic homotopy thanks to the following theorem, due to Anderson-Davis~\cite{anderson/davis:1973} for $p=2$ and Miller-Wilkerson~\cite{miller/wilkerson:1981} for $p>2$.
\begin{theorem}\label{theorem: vanishing lines}
Let $M$ be a connected $A$-module that is free over $A_{k}$. Then $\operatorname{Ext}_A(M, {\mathbb F}_p)$ has a vanishing line of slope that is strictly smaller than $\frac{1}{|v_{k}|}$ and an intercept that is bounded above by a number that depends only on $k$. This bound is a slowly growing function of $k$.
\end{theorem}

\begin{corollary}\label{corollary: trivial}
Suppose $E$ is a spectrum for which $H^*(E)$ is free over $A_{l-1}$. Then $E$ and $\Omega^\infty E$ have trivial $v_k$-periodic homotopy for $k < l$.
\end{corollary}
\begin{proof}
It is enough to show that for a complex $V_k$ that is strongly of type $k$, $v_k^{-1}\pi_*(E; V_k)=0$. Since $H^*(E)$ is free over $A_{l-1}$, it follows that $H^*(E^{V_k})$ is free over $A_{l-1}$. By theorem~\ref{theorem: vanishing lines}, the ASS for $\pi_*(E^{V_k})$ has a vanishing line of slope smaller than $\frac{1}{|v_{l-1}|}$, and therefore smaller than $\frac{1}{|v_{k}|}$. By proposition~\ref{proposition: strong type k} and subsequent comments, multiplication by a self map of $V_k$ is represented, on the level of second page of the ASS, by multiplication by an element on a line of slope $\frac{1}{|v_k|}$, which is larger than the slope of the vanishing line. Therefore, every element of $\pi_*(E^{V_k})$ is annihilated by some power of the self-map of $V_k$. It follows that inverting the self map kills everything in the homotopy groups.
\end{proof}
\begin{corollary}\label{corollary: nontrivial layers}
If $X$ is an odd sphere, then the only layers of Goodwillie tower of the identity evaluated at $X$ that are non-trivial in $v_k$-periodic homotopy are $D_1I(X), D_pI(X), D_{p^2}I(X), \ldots, D_{p^k}I(X)$.
\end{corollary}
\begin{proof}
Everything is localized at $p$, and by theorem~\ref{theorem: primary}, $D_{n}I(X)$ is non-trivial at $p$ only if $n$ is a power of $p$. We need to show that if $l>k$ then $D_{p^l}I(X)$ is trivial in $v_k$-periodic homotopy. Recall that
\[ D_{p^l}I(X)\simeq \Omega^\infty\left(\left(\partial_{p^l}I \wedge S^{p^l}\right)_{h\Sigma_{p^l}}\right). \]

By Proposition~\ref{proposition: A_k free}, $H^*\left(\left(\partial_{p^l}I \wedge S^{p^l}\right)_{h\Sigma_{p^l}}\right)$ is free over ${\mathcal A}_{l-1}$. The result follows by corollary~\ref{corollary: trivial}.
\end{proof}

Corollary~\ref{corollary: nontrivial layers} suggests that the map $X\to P_{p^k}I(X)$ induces an isomorphism on $v_k$-periodic homotopy groups. This conclusion is not automatic, because inverting $v_k$ does not always commute with inverse homotopy limits, but it turns out to be true in this case. So we can now sketch the proof of Theorem~\ref{thm:vk-periodic}.

\begin{proof}[Proof of Theorem~\ref{thm:vk-periodic}]
Let $k$ be fixed. Let $F_k(X)$ be the homotopy fibre of the map $X\to P_{p^k}I(X)$. For $n\ge k$ let us define $\overline{P_{p^n}}(X)$ to be the homotopy fibre of the map $P_{p^n}I(X)\to P_{p^k}I(X)$. Note that the homotopy fibre of the map $\overline{P_{p^n}}(X) \to \overline{P_{p^{n-1}}}(X)$ is $D_{p^n}I(X)$. Choose a complex $V_k$ that is strongly of type $k$ with a self map $\nu_k\colon \Sigma^{d|v_k|} V_k \to V_k$ of Adams filtration $d$, as per proposition~\ref{proposition: strong type k}. There is a tower of the following form, where $F_k(X)^{V_k}$ is the homotopy inverse limit of the top row:
\[ \begin{diagram} \dgARROWLENGTH=1em
  \node{F_k(X)^{V_k} \cdots} \arrow{e} \node{\overline{P_{p^n}}(X)^{V_k}} \arrow{e} \node{\overline{P_{p^{n-1}}}(X)^{V_k}} \arrow{e} \node{\cdots} \arrow{e} \node{\overline{P_{p^{k+1}}}(X)^{V_k}} \\
  \node[2]{D_{p^n}I(X)^{V_k}} \arrow{n} \node{D_{p^{n-1}}I(X)^{V_k}} \arrow{n} \node[2]{D_{p^{k+1}}I(X)^{V_k}.} \arrow{n,r}{\sim}
\end{diagram} \]
We need to prove that $v_k^{-1}\pi_*(F_k; V_k)$ is zero. Let $\alpha\in\pi_*(F_k^{V_k})$. We need to show that some power of the self map $\nu_k$ annihilates $\alpha$. Let us say that $\alpha$ has height $\ge m$ if the image of $\alpha$ in $\pi_*\left(\overline{P_{p^n}}^{V_k}\right)$ is zero for $n<m$. Suppose $\alpha$ has height $\ge m$. Then the image of $\alpha$ in $\pi_*\left(\overline{P_{p^m}}^{V_k}\right)$ is in the image of some element of $\pi_*\left(D_{p^m}I(X)^{V_k}\right)$. We know that $\nu_k$ acts along a line of higher slope than the vanishing line of the ASS for the underlying spectrum of $D_{p^m}I(X)^{V_k}$. Because of this, some power $\nu_k^l$ of $\nu_k$ annihilates this element of $\pi_*\left(D_{p^m}I(X)^{V_k}\right)$. This means that $v_k^l(\alpha)$ has height $\ge m+1$. Repeating the process, one finds that for every $n>k$, there exists some $l_n$ for which $v_k^{l_n}$ has height $\ge n$. Let $l_n$ be the lowest such $l$. Using Theorem~\ref{theorem: vanishing lines}, it is not very difficult to write down an effective upper bound on the topological dimension of $\nu_k^{l_n}(\alpha)$ as a function of $n$. Crucially, it is not very difficult to show that this function grows slower than the connectivity of $D_{p^n}I(X)^{V_k}$ (the calculation is done in~\cite{arone/mahowald:1999}). It follows that there exists an $n$ for which $\nu_k^{l_n}(\alpha)$ has lower topological dimension than the connectivity of $D_{p^n}I(X)^{V_k}$. It follows that the image of $\nu_k^{l_n}(\alpha)$ in $\pi_*\left(\overline{P_{p^m}}^{V_k}\right)$ is zero for all $m$. It follows that $\nu_k^{l_n}(\alpha)=0$.
\end{proof}
Theorem~\ref{thm:vk-periodic} says that the unstable $v_k$-periodic homotopy type of an odd sphere can be resolved into ($k+1$) stable homotopy types. The case $k=0$ is essentially~\ref{corollary: serre}.

\subsection*{A Bousfield-Kuhn functor reformulation}
Theorem~\ref{thm:vk-periodic} has a more modern reformulation in terms of the Bousfield-Kuhn functor. Let us recall what this functor is. Choose a $v_k$ self map $\Sigma^{d|v_k|}V_k \to V_k$. Use it to form a direct system of spectra
\[ \Sigma^{\infty} V_k \to \Omega^{d|v_k|}\Sigma^{\infty} V_k \to \cdots . \]
Let $T(k)$ be the homotopy colimit of this system, and $L_{T(k)}$ be the Bousfield localization with respect to $T(k)$. It follows from the uniqueness part of the Periodicity Theorem that the functor $L_{T(k)}$ does not depend on the choice of $V_k$ or the self map.

The Bousfield-Kuhn functor $\Phi_k$ is a functor from pointed spaces to spectra, whose main property is that there is an equivalence $\Phi_k(\Omega^\infty E)\simeq L_{T(k)}(E)$. Here $E$ is any spectrum, the equivalence is natural in $E$. The functor $\Phi_k$ is constructed as an inverse homotopy limit
\[ \Phi_k(X)=\operatorname{holim}_{\alpha}v_k^{-1} X^{V_k^\alpha}, \]
where $\{V_k^\alpha\}$ is a direct system of complexes of type $k$ with certain properties. See~\cite{kuhn:2008} for more details. The following is now an immediate corollary of Theorem~\ref{thm:vk-periodic}.
\begin{theorem}\label{theorem: bousfield-kuhn}
When $X$ is an odd sphere, the map $X\to P_{p^k}I(X)$ becomes an equivalence after applying the Bousfield-Kuhn functor $\Phi_k$.
\end{theorem}

\subsection*{The case of even-dimensional spheres, and beyond} There is a version of Theorems~\ref{thm:vk-periodic} and~\ref{theorem: bousfield-kuhn} that holds for even-dimensional spheres:
\begin{theorem}\cite[Theorems 4.4 and 4.5]{arone/mahowald:1999}  \label{theorem: even sphere}
Fix a prime $p$ and localize everything at $p$. Let $X$ be an even-dimensional sphere. Then $D_nI(X)\simeq *$ unless $n$ is a power of $p$ or twice a power of $p$. The map
\[ X\to P_{2p^k}I(X) \]
is a $v_k$-periodic equivalence.
\end{theorem}
The easiest way to prove this theorem that we know is to use the EHP sequence.
\[ X \mapsto \Omega\Sigma X \to \Omega\Sigma (X\wedge X). \]
This is a fibration sequence if $X$ is an odd-dimensional sphere, and one can show that it induces a fibration sequence of Taylor towers. Given this, it is easy to deduce Theorem~\ref{theorem: even sphere} from Theorem~\ref{thm:vk-periodic}.
\begin{remark} The connection between the Goodwillie tower and the EHP sequence was investigated much more deeply by M. Behrens at~\cite{behrens:2012}.
\end{remark}

Theorems~\ref{thm:vk-periodic} and~\ref{theorem: even sphere} tell us that the behavior of the Taylor tower of the identity in $v_k$-periodic homotopy has two non-obvious properties when evaluated at spheres:
\begin{enumerate}
\item It is finite. \label{finite}
\item It converges. \label{converges}
\end{enumerate}
It is therefore reasonable to ask if there exist other spaces for which the Taylor tower of the identity has these properties. Regarding the finiteness property, it seems clear that spheres (or at best homology spheres) are the only spaces for which the Taylor tower is finite in $v_k$-periodic homotopy. For example, it is easy to show that the only spaces for which the tower is rationally finite are rational homology spheres. On the other hand, there do exist other spaces for which the tower converges in $v_k$-periodic homotopy. Such results were obtained by Behrens-Rezk~\cite[section 8]{behrens/rezk:2017} and by Heuts~\cite{heuts:2018a}. For example, convergence holds for products of spheres and special unitary groups $SU(k)$. It would be interesting to characterize the spaces for which $v_k$-periodic convergence holds.

\subsection*{The case $X=S^1$. Theorems of Behrens and Kuhn} We noted above that there is a relationship between the layers of the Taylor tower of the identity evaluated at $S^1$, and the subquotients of the filtration of the Eilenberg-MacLane spectrum $H{\mathbb Z}=\mathrm{Sp}^\infty(S^0)$ by the symmetric powers of the sphere spectrum. More precisely, if $n$ is not a power prime, then both layers are trivial, and if $n=p^k$ then there is an equivalence
\[
D_{p^k}I(S^1)\simeq \Omega^{\infty+2k-1}\mathrm{Sp}^{p^k}(S^0)/\mathrm{Sp}^{p^k-1}(S^0).
\]
In fact, there is a deeper connection between the two filtrations. Roughly speaking, one can recast each filtration as a ``chain complex'' of infinite loop spaces or spectra, and each chain complex can serve as a kind of contracting homotopy for the other one. This  implies, in particular, that each filtration is trivial in the sense that its homotopy spectral sequence collapses at the second page. The triviality result for the symmetric powers filtration used to be known as the Whitehead conjecture. It was proved by Kuhn~\cite{kuhn:1982} at the prime $2$ and by Kuhn and Priddy~\cite{kuhn/priddy:1985} at odd primes. The triviality result for the Goodwillie tower of the identity, and the connection between the two filtrations was proved by Behrens~\cite{behrens:2011} at the prime $2$ and in another way by Kuhn~\cite{kuhn:2015} at all primes.

\section{Operads and Tate Data: the Classification of Taylor Towers} \label{sec:operads}

The layers of a Taylor tower (say, of a functor $F: \mathscr{T}op_* \to \mathscr{T}op_*$) are homogeneous functors classified by the spectra $\partial_*F = (\partial_nF)_{n \geq 1}$, along with the action of $\Sigma_n$ on $\partial_nF$, i.e. a \emph{symmetric sequence} of spectra. As we have seen, however, further information is needed to encode the full Taylor tower, and hence under convergence conditions, the functor $F$ itself.

Roughly speaking, there are two approaches to understanding this extra information: (1) inductive techniques based on the delooped fibre sequences
\[ P_nF \to P_{n-1}F \to \Omega^{-1}D_nF; \]
and (2) analysis of operad/module structures on the symmetric sequence $\partial_*F$ in its entirety, based on a chain rule philosophy for the calculus of functors.

We saw approach (1) applied to $P_2I$ in Example~\ref{ex:id} and, for functors from spectra to spectra, in McCarthy's Theorem (\ref{thm:McCarthy}). Here we focus on (2).

\subsection*{Chain rules in functor calculus}

The first version of a `Chain Rule' for the calculus of functors was proved by Klein and Rognes in \cite{klein/rognes:2002}. The simplest version of this states the following: for reduced functors $F,G: \mathscr{T}op_* \to \mathscr{T}op_*$, we have
\[ \partial_1(FG) \simeq \partial_1F \wedge \partial_1G. \]\index{Chain Rule}
More generally, Klein and Rognes provided a formula for the first derivative of $FG$ at an arbitrary space $X$, in terms of the first derivatives of $G$ (at $X$) and $F$ (at $G(X)$).

For higher derivatives, the Chain Rule for functors of spectra is much simpler than that for spaces. Suppose first that $F,G: {\mathscr{S}p} \to {\mathscr{S}p}$ are given by the formulas
\[ F(X) \simeq \bigvee_{k = 1}^{\infty} (\partial_kF \wedge X^{\wedge k})_{h\Sigma_k}, \quad G(X) \simeq \bigvee_{l = 1}^{\infty} (\partial_lG \wedge X^{\wedge l})_{h\Sigma_l}. \]
A simple calculation then shows that
\[ \partial_n(FG) \simeq \bigvee_{\textrm{partitions of \{1,\dots,n\}}} \partial_kF \wedge \partial_{n_1}G \wedge \dots \wedge \partial_{n_k}G \]
where $n_1,\dots,n_k \geq 1$ are the sizes of the terms in a given partition. This is also the formula for the composition product of symmetric sequences, and, more succinctly, we can write
\begin{equation} \label{eq:chain} \partial_*(FG) \simeq \partial_*F \circ \partial_*G. \end{equation}
The second author proved in \cite{ching:2010} that (\ref{eq:chain}) holds for any reduced functors $F,G: {\mathscr{S}p} \to {\mathscr{S}p}$ where $F$ preserves filtered colimits.

For functors of based spaces (or more general $\infty$-categories), the formula (\ref{eq:chain}) does not hold, but there is nonetheless a natural map
\[ l: \partial_*F \circ \partial_*G \to \partial_*(FG) \]
making $\partial_*$, at least up to homotopy, into a lax monoidal functor from a suitable $\infty$-category of functors to the $\infty$-category of symmetric sequences. Specializing the map $l$ to the case when one or both of $F,G$ is the identity, we obtain a number of important consequences. The following results were proved in~\cite{arone/ching:2011} for functors between based spaces and spectra, and (with the exception of the final statement of the chain rule, though a version of this appears in \cite[6.3]{lurie:2017}) in~\cite{ching:2018} for arbitrary $\infty$-categories.

\begin{theorem} \label{thm:operad}
The derivatives $\partial_*I_{\mathscr{C}}$ of the identity functor on an $\infty$-category $\mathscr{C}$ have a canonical operad structure, and for an arbitrary functor $F: \mathscr{C} \to \mathscr{D}$, the derivatives $\partial_*F$ form a $(\partial_*I_{\mathscr{D}},\partial_*I_{\mathscr{C}})$-bimodule. Moreover, if $F$ preserves filtered colimits and $G: \mathscr{B} \to \mathscr{C}$ is reduced, then there is an equivalence (of bimodules):
\[ \partial_*(FG) \simeq \partial_*F \circ_{\partial_*I_{\mathscr{C}}} \partial_*G \]
where the right-hand side involves a (derived) relative composition product of bimodules over the operad $\partial_*I_{\mathscr{C}}$.
\end{theorem}

\begin{remark}
When $\mathscr{C} = \mathscr{T}op_*$, the $\infty$-category of based spaces, the operad $\partial_*I_{\mathscr{T}op_*}$ from Theorem~\ref{thm:operad} can be viewed as the analogue in stable homotopy theory of the operad encoding the structure of a Lie algebra. This perspective can be justified in a number of ways. Firstly, it follows from Johnson's calculation that taking homology groups of the spectra $\partial_*I_{\mathscr{T}op_*}$ recovers precisely the ordinary Lie operad (up to a shift in degree). A deeper connection is given by viewing $\partial_*I_{\mathscr{T}op_*}$ as an example of bar-cobar (or Koszul) duality for operads of spectra. Ginzburg and Kapranov developed the theory of bar-cobar duality for differential-graded operads in~\cite{ginzburg/kapranov:1994}, and identified the Lie operad as the dual of the commutative cooperad. The following result was proved in~\cite{ching:2005} (or in the given form in~\cite{arone/ching:2011}) and justifies viewing $\partial_*I_{\mathscr{T}op_*}$ as a version of the Lie operad.
\end{remark}

\begin{theorem} \label{thm:koszul}
The operad $\partial_*I_{\mathscr{T}op_*}$ is equivalent to the cobar construction, or (derived) Koszul dual, of the cooperad $\partial_*(\Sigma^\infty \Omega^\infty)$, which itself can be identified with the commutative cooperad of spectra.
\end{theorem}

\begin{remark}
Theorem~\ref{thm:koszul} can actually be understood quite easily from the point of view of calculus. For simply-connected $X$, the adjunction $(\Sigma^\infty,\Omega^\infty)$ determines an equivalence
\[ X \; \tilde{\longrightarrow} \; \operatorname{Tot}(\Omega^\infty(\Sigma^\infty \Omega^\infty)^\bullet\Sigma^\infty X) \]
which connects the identity functor to a cobar construction on the comonad $\Sigma^\infty \Omega^\infty$. Taking derivatives of each side, and applying the chain rule for spectra from (\ref{eq:chain}) we recover \ref{thm:koszul} with a little work. Similar arguments were used in \cite{arone/ching:2011} to understand the bimodule structures on the derivatives of functors to/from based spaces.
\end{remark}

\begin{remark} \label{rem:lurie}
A version of Theorem~\ref{thm:koszul} seems likely to be valid for the identity functor on an arbitrary $\mathscr{C}$. Lurie constructs in \cite[6.3.0.14]{lurie:2017} a cooperad that represents the derivatives $\partial_*(\Sigma^\infty_{\mathscr{C}}\Omega^\infty_{\mathscr{C}})$. (In Lurie's language, this object is actually a ``stable corepresentable $\infty$-operad'', but the connection with cooperads is explained in \cite[6.3.0.12]{lurie:2017}.) A similar proof to that of \ref{thm:koszul} should imply that the cobar construction on this cooperad recovers the operad $\partial_*I_{\mathscr{C}}$.
\end{remark}

\subsection*{Heuts's theorem on spectral Lie algebras in chromatic homotopy theory}

The calculations of Arone and Mahowald~\cite{arone/mahowald:1999} described above already show that the Taylor tower of the identity functor on based spaces has interesting structure when viewed through the lens of $v_k$-periodic homotopy theory. Recent work of Heuts~\cite{heuts:2018a} has further developed this connection, taking the Lie operad structure on $\partial_*I_{\mathscr{T}op_*}$ into account.

First recall that one of Quillen's models for rational homotopy theory is in terms of Lie algebras. Specifically, he constructs in~\cite{quillen:1969} a Quillen equivalence between simply-connected rational spaces and $0$-connected differential-graded rational Lie algebras. Heuts's work extends Quillen's to higher chromatic height in the following way.

\begin{theorem}[Heuts] \label{thm:heuts}
Fix a prime $p$ and positive integer $k$. Let $\mathscr{M}^f_k$ denote the $\infty$-category obtained from that of $p$-local based spaces by inverting those maps that induce an equivalence on $v_k$-periodic homotopy groups. Then there is an equivalence of $\infty$-categories
\[ \mathscr{M}^f_k \simeq {\mathscr{S}p}_{T(k)}(\partial_*I_{\mathscr{T}op_*}) \]
between $\mathscr{M}^f_k$ and the category of $T(k)$-local spectra with an algebra structure over the operad $\partial_*I_{\mathscr{T}op_*}$. Moreover, to a space $X$ this equivalence assigns a $\partial_*I_{\mathscr{T}op_*}$-algebra whose underlying spectrum is given by applying the Bousfield-Kuhn functor $\Phi_k$ to $X$.
\end{theorem}

Theorem~\ref{thm:heuts} should be compared to Kuhn's Theorem~\ref{thm:Kuhn} which also described a simplification to Goodwillie calculus that appears in the presence of chromatic localization. Both results illustrate the general principle that analysis in Goodwillie calculus typically comprises two pieces: (1) an operadic part related to the derivatives of the identity functor, and (2) something related to the Tate spectrum construction, which vanishes chromatically. Notice that in Kuhn's result, the operadic part is absent because the derivatives of the identity on ${\mathscr{S}p}$ form the trivial operad.

\subsection*{Tate data and the classification of Taylor towers}

We have already seen in McCarthy's Theorem \ref{thm:McCarthy} that the extra information for reconstructing the Taylor tower of a functor $F: {\mathscr{S}p} \to {\mathscr{S}p}$ can be described in terms of Tate spectrum constructions for the $\Sigma_n$-action on $\partial_nF$. For functors $F: \mathscr{T}op_* \to \mathscr{T}op_*$, we instead think of the Taylor tower information being given by a combination of such ``Tate data'' with the $\partial_*I$-bimodule structure of Theorem~\ref{thm:operad}.

To understand this perspective, we use the fact the functor $\partial_*$ (when viewed with values in $\partial_*I$-bimodules) admits a right adjoint which we denote $\Psi$, i.e. there is an adjunction
\begin{equation} \label{eq:adj} \partial_*: \operatorname{Exc}^*_n(\mathscr{C},\mathscr{D}) \rightleftarrows \operatorname{Bimod}_{\leq n}(\partial_*I_{\mathscr{D}},\partial_*I_{\mathscr{C}}) : \Psi \end{equation}
between the $\infty$-category of reduced $n$-excisive functors $\mathscr{C} \to \mathscr{D}$, and the $\infty$-category of $n$-truncated bimodules over the derivatives of the identity on $\mathscr{C}$ and $\mathscr{D}$. (The right adjoint was denoted $\Phi$ in~\cite{arone/ching:2015} but we use $\Psi$ here to avoid confusion with the Bousfield-Kuhn functor.) In \cite{arone/ching:2015}, we proved the following:

\begin{theorem} \label{thm:class}
The adjunction $(\partial_*,\Psi)$ of (\ref{eq:adj}) is comonadic. In particular, an $n$-excisive functor $F: \mathscr{C} \to \mathscr{D}$ is classified by the bimodule $\partial_*F$ together with an action of the comonad $\mathbf{C} := \partial_*\Psi$.
\end{theorem}

Understanding the full structure on the derivatives of a functor $F: \mathscr{C} \to \mathscr{D}$ thus involves a calculation of the comonad $\mathbf{C}$ of Theorem~\ref{thm:class}. This is most easily described in the case of functors $F: \mathscr{T}op_* \to {\mathscr{S}p}$ where we obtain the following consequence.

\begin{theorem} \label{thm:spacesspectra}
The Taylor tower of $F: \mathscr{T}op_* \to {\mathscr{S}p}$ is determined by the right $\partial_*I$-module structure on $\partial_*F$ together with (suitably compatible) lifts $\psi_{n_1,\dots,n_k}$ of the form
\[ \begin{diagram}
  \node[2]{\operatorname{Map}(\partial_{n_1}I \wedge \dots \wedge \partial_{n_k}I, \partial_nF)_{h\Sigma_{n_1} \times \dots \times \Sigma_{n_k}}} \arrow{s,r}{N} \\
  \node{\partial_kF} \arrow{ne,t,..}{\psi_{n_1,\dots,n_k}} \arrow{e,b}{\phi_{n_1,\dots,n_k}} \node{\operatorname{Map}(\partial_{n_1}I \wedge \dots \wedge \partial_{n_k}I, \partial_nF)^{h\Sigma_{n_1} \times \dots \times \Sigma_{n_k}}}
\end{diagram} \]
where $N$ is the norm map from homotopy orbits to homotopy fixed points, and $\phi_{n_1,\dots,n_k}$ is the map associated to the right $\partial_*I$-module structure.
\end{theorem}
Such lifts make $\partial_*F$ into what we call a \emph{divided power right $\partial_*I$-module}.\index{divided power module}

The data of the lifts $\psi_{n_1,\dots,n_k}$ in Theorem~\ref{thm:spacesspectra} can be reframed as choices of nullhomotopies for maps
\[ \partial_kF \to \operatorname{Map}(\partial_{n_1}I \wedge \dots \wedge \partial_{n_k}I,\partial_nF)^{t\Sigma_{n_1} \times \dots \times \Sigma_{n_k}} \]
into the corresponding Tate construction. We think of this as the \emph{Tate data} corresponding to the Taylor tower of the functor $F$.

\begin{problem}
It is still unclear how to describe the structure on the derivatives of a functor $F: \mathscr{T}op_* \to \mathscr{T}op_*$ as explicitly as that in Theorem~\ref{thm:spacesspectra}. The comonad guaranteed by Theorem~\ref{thm:class} is hard to understand in this case. In \cite{arone/ching:2015} we gave a concrete description of this structure for $3$-excisive functors, but a more general picture was too elusive.
\end{problem}

\subsection*{Vanishing Tate data and applications}

Another way to see how the Tate spectrum construction comes up is via the unit map $\eta: P_nF \to \Psi \partial_{\leq n}F$ of the adjunction (\ref{eq:adj}). The right adjoint $\Psi$ can be written in terms of mapping spaces for the $\infty$-category of bimodules. This leads to the existence of the following diagram for, for example, a functor $F: \mathscr{T}op_* \to \mathscr{T}op_*$:
\begin{equation} \label{eq:tate} \begin{diagram}
  \node{\Omega^\infty(\partial_nF \wedge X^{\wedge n})_{h\Sigma_n}} \arrow{e,t}{N} \arrow{s} \node{\Omega^\infty(\partial_nF \wedge X^{\wedge n})^{h\Sigma_n}} \arrow{s} \\
  \node{P_nF(X)} \arrow{e,t}{\eta} \arrow{s} \node{\operatorname{Bimod}_{\partial_*I}(\partial_*(\operatorname{Hom}(X,-)),\partial_{\leq n}F)} \arrow{s} \\
  \node{P_{n-1}F(X)} \arrow{e,t}{\eta} \node{\operatorname{Bimod}_{\partial_*I}(\partial_*(\operatorname{Hom}(X,-)),\partial_{\leq n-1}F)}
\end{diagram} \end{equation}
where the columns are fibration sequences, and $\operatorname{Bimod}_{\partial_*I}(-,-)$ denotes the space of maps of $\partial_*I$-bimodules between the given derivatives. Here $\partial_{\leq n}F$ denotes the \emph{$n$-truncation} of a $\partial_*I$-bimodule, given by setting all terms in degree larger than $n$ to be the trivial spectrum.

We think of (\ref{eq:tate}) as a generalization of McCarthy's Theorem~\ref{thm:McCarthy}, and there is a similar picture for functors from spectra to spectra that reduces to McCarthy's result. For functors $F: \mathscr{T}op_* \to {\mathscr{S}p}$, we can similarly deduce the following version, which is from~\cite[4.17]{arone/ching:2015}.

\begin{theorem}
Let $F: \mathscr{T}op_* \to {\mathscr{S}p}$ be a reduced functor that preserved filtered colimits. Then there is a pullback square of the form
\[ \begin{diagram}
  \node{P_nF(X)} \arrow{e} \arrow{s} \node{(\partial_nF \wedge X^{\wedge n}/\Delta^n X)_{h\Sigma_n}} \arrow{s} \\
  \node{P_{n-1}F(X)} \arrow{e} \node{(\partial_nF \wedge \Sigma\Delta^n X)_{h\Sigma_n}}
\end{diagram} \]
where $\Delta^n X$ is the fat diagonal inside $X^{\wedge n}$.
\end{theorem}

All the results mentioned here have similar consequences to McCarthy's Theorem in cases where the Tate spectra vanish. In such situations, the Taylor tower of a functor is completely determined by the relevant module or bimodule structure.

\begin{proposition} \label{prop:bimod}
Suppose $F: \mathscr{T}op_* \to \mathscr{T}op_*$ preserves filtered colimits and $X$ has the property that
\[ (\partial_nF \wedge X^{\wedge n})^{t\Sigma_n} \simeq * \]
for $2 \leq n \leq N$. Then
\[ P_NF(X) \simeq \operatorname{Bimod}_{\partial_*I}(\partial_*\operatorname{Hom}(X,-),\partial_{\leq N}F). \]
For $F: \mathscr{T}op_* \to {\mathscr{S}p}$ satisfying the same assumptions
\[ P_NF(X) \simeq \operatorname{RMod}_{\partial_*I}(\partial_*\Sigma^\infty\operatorname{Hom}(X,-),\partial_{\leq N}F) \]
where $\operatorname{RMod}_{\partial_*I}(-,-)$ is the mapping spectrum for the $\infty$-category of right $\partial_*I$-modules.
\end{proposition}

Since the Tate construction vanishes for any rational spectrum, we see that the hypothesis of Proposition~\ref{prop:bimod} holds if either:
\begin{itemize}
  \item $\partial_nF$ is a rational spectrum for $2 \leq n \leq N$;
  \item $X$ is a rational (simply-connected) topological space.
\end{itemize}
In particular, if $F$ is a functor either to or from the category of rational based spaces, then the Taylor tower of $F$ is given by the formula in Proposition~\ref{prop:bimod}.

Proposition~\ref{prop:bimod} allows us to extend Kuhn's Theorem~\ref{thm:Kuhn} to functors from $\mathscr{T}op_*$ to ${\mathscr{S}p}$.

\begin{proposition}
Let $F$ be a functor from $\mathscr{T}op_*$ to $T(n)$-local spectra. Then there is an equivalence
\[ P_NF(X) \simeq  \operatorname{RMod}_{\partial_*I}(\mathbb{D}(X^{\wedge *}/\Delta^*X),\partial_{\leq N}F). \]
Similarly, if $F$ is a functor from $\mathscr{T}op_*$ to $\mathscr{T}op_*$, and $V_k$ is a finite complex of type $k$, then there is an equivalence
\[ P_Nv_k^{-1}F(X)^{V_k} \simeq  \operatorname{RMod}_{\partial_*I}(\mathbb{D}(X^{\wedge *}/\Delta^*X),v_k^{-1}\partial_{\leq N}F^{V_k}). \]
The spectra $\mathbb{D} X^{\wedge *}/\Delta^*X$ form a right $\partial_*I$-module, by identifying them as the derivatives of the functor $\Sigma^\infty \operatorname{Hom}(X,-): \mathscr{T}op_* \to {\mathscr{S}p}$.
\end{proposition}

It seems less straightforward to obtain an analogous result for functors from spaces to spaces than in the spectrum-valued case, yet Heuts's Theorem~\ref{thm:heuts} simplifies this to the case of functors between spectral Lie algebras.

\subsection*{Functors from spaces to spectra and modules over the little disc operads}

There is a close connection between the spectral Lie operad $\partial_*I$ and the little disc operads. Write $E_n$ for the operad of spectra formed by taking suspension spectra of the terms in the little $n$-discs operad of May~\cite{may:1972}, and write $KE_n$ for the (derived) Koszul dual of $E_n$ in the sense of~\cite{ching:2012}. Then $\partial_*I$ can be expressed as the inverse limit of the sequence of Koszul duals:
\[ \partial_*I \to \dots \to KE_n \to KE_{n-1} \to \dots \to KE_1. \]
It is conjectured that $KE_n$ is equivalent as an operad to a desuspension of $E_n$ itself. This is proved in the context of chain complexes by Fresse~\cite{fresse:2011}.

The connection between $\partial_*I$ and the sequence $KE_n$ has powerful consequences. Any divided power module over $KE_n$ determines, by pulling back along the map $\partial_*I \to KE_n$, a divided power module over $\partial_*I$. Moreover, the terms in the operad $KE_n$ have free symmetric group actions, so the norm map for the $\Sigma_k$-action on $\operatorname{Map}(KE_n(k),X^{\wedge k})$ is an equivalence. It follows that any right $KE_n$-module has a unique divided power structure, and hence determines a divided power $\partial_*I$-module.

This allows us to ask the following question: for a given functor $F: \mathscr{T}op_* \to {\mathscr{S}p}$, is the Taylor tower of $F$ determined by a right $KE_n$-module structure on $\partial_*F$?

The following result of \cite{arone/ching:2017} gives a criterion for this to be the case.

\begin{theorem} \label{thm:framed}
A polynomial functor $F: \mathscr{T}op_* \to {\mathscr{S}p}$ is determined by a $KE_n$-module structure on $\partial_*F$ if and only if $F$ is the left Kan extension of a functor $f\mathscr{M}an_n \to {\mathscr{S}p}$ along the inclusion $f\mathscr{M}an_n \to \mathscr{T}op_*$, where $f\mathscr{M}an_n$ is the subcategory of $\mathscr{T}op_*$ consisting of certain `pointed framed $n$-dimensional manifolds' (that is, one-point compactifications of framed $n$-manifolds) and `pointed framed embeddings'.
\end{theorem}

In the next section we will see an application of \ref{thm:framed} to the Taylor tower of algebraic $K$-theory of spaces.

\begin{remark}
The category $f\mathscr{M}an_n$ in Theorem~\ref{thm:framed} is related to the ``zero-pointed manifolds'' of Ayala and Francis~\cite{ayala/francis:2014}, and this result suggests deeper connections between Goodwillie calculus and factorization homology that are yet to be explored.
\end{remark}

\section{Applications and Calculations in Algebraic K-Theory}

Much of Goodwillie's initial motivation for developing the calculus of functors came from algebraic $K$-theory, and aside from the identity functor, most of the calculations and applications of calculus have been in this area, largely by Randy McCarthy and coauthors. We review here what is known about the Taylor tower of algebraic $K$-theory both in the context of spaces and ring spectra.

\subsection*{Algebraic K-theory of spaces}

Let $A: \mathscr{T}op \to {\mathscr{S}p}$ denote Waldhausen's functor calculating the algebraic $K$-theory\index{algebraic $K$-theory of spaces} of a topological space $X$, i.e. of the category of spaces over and under $X$; see \cite{waldhausen:1985}. For calculus to be at all relevant to the study of the functor $A$, we first have to know that the Taylor tower converges for some spaces $X$.

\begin{theorem}[Goodwillie \cite{goodwillie:1991}]
The functor $A: \mathscr{T}op \to {\mathscr{S}p}$ is $1$-analytic. Thus the Taylor tower of $A$ converges on simply-connected spaces.
\end{theorem}

Goodwillie's initial application of this result, however, was not to the convergence of the Taylor tower, but to the cyclotomic trace map $\tau$ from $A(X)$ to the \emph{topological cyclic homology} $TC(X)$ of B\"{o}kstedt, Hsiang and Madsen \cite{bokstedt/hsiang/madsen:1993}.

\begin{theorem}[B\"{o}ksted, Carlsson, Cohen, Goodwillie, Hsiang, Madsen~\cite{bokstedt/carlsson/cohen/goodwillie/hsiang/madsen:1996}] \label{thm:BCCGHM}
For a simply-connected finite CW-complex $X$, there is a pullback square
\[ \begin{diagram}
  \node{A(X)} \arrow{e,t}{\tau} \arrow{s} \node{TC(X)} \arrow{s} \\
  \node{A(*)} \arrow{e,t}{\tau} \node{TC(*).}
\end{diagram} \]
\end{theorem}
The main step in the proof of Theorem~\ref{thm:BCCGHM} is to show that the cyclotomic trace $\tau$ induces equivalences between the first derivatives of $A$ and $TC$ (at an arbitrary space $X$). The general theory of calculus then implies that the fibre of $\tau$ is ``locally constant'', i.e. takes a suitably connected map of spaces to an equivalence of spectra.

The first derivative of $A$ at the space $X$ can be thought of as a parameterized spectrum over the space $X$ with fibre over $x \in X$ given by
\[ \partial_1A(X)_x \simeq \Sigma^\infty(\Omega_xX)_+ \]
In \cite{goodwillie:2003}, Goodwillie generalizes this formula to higher derivatives. The $n$-th derivative is a spectrum parameterized over $X^n$ and for simplicity, we will only describe the case $X = *$, where
\begin{equation} \label{eq:derA} \partial_nA \simeq \Sigma^\infty(\Sigma_n/C_n)_+ \end{equation}
and where $C_n$ is an order $n$ cyclic subgroup of the symmetric group $\Sigma_n$.

What can we now say about the Taylor tower of $A$? Firstly, a choice of basepoint for a space $Y$ determines a splitting
\[ A(Y) \simeq A(*) \times \tilde{A}(Y) \]
where $\tilde{A}: \mathscr{T}op_* \to {\mathscr{S}p}$ is the corresponding reduced functor. Next, Waldhausen's splitting result \cite{waldhausen:1985} implies that the $1$-excisive approximation to $A$ splits off too, so we have
\[ A(Y) \simeq A(*) \times \Sigma^\infty Y \times \tilde{\operatorname{Wh}}(Y) \]
where $\tilde{\operatorname{Wh}}(Y)$ is the ``reduced Whitehead spectrum'' of $Y$, which contains all of the higher degree information from the Taylor tower of $A$.

The formula (\ref{eq:derA}) implies the following simple calculation of the layers of the Taylor tower for $A$: for $n \geq 2$ we have (using the chosen basepoint for $Y$):
\[ D_nA(Y) \simeq (\Sigma^\infty Y)^{\wedge n}_{hC_n}. \]
How are these layers attached to each other to form $\tilde{\operatorname{Wh}}(Y)$?

Recall that Theorem~\ref{thm:framed} gave us conditions for the Taylor tower of a functor $F: \mathscr{T}op_* \to {\mathscr{S}p}$ to be determined by a $KE_n$-module structure on $\partial_*F$. The following result of \cite{arone/ching:2017} applies this to the functor $A$.

\begin{theorem} \label{thm:A}
The divided power module structure on $\partial_*A$, and hence the Taylor tower of $A$, is determined by a certain $KE_3$-module structure on $\partial_*A$.
\end{theorem}

Note that the $KE_3$-module structure on $\partial_*A$ pulls back to a $KE_n$ structure for any $n > 3$, but we do not know if in fact it comes from a $KE_2$-, or $KE_1$-module.

\begin{problem}
Extract from the proof of Theorem~\ref{thm:A} an explicit description of the $KE_3$-module structure on $\partial_*A$, and a corresponding construction of the Taylor tower of $A$.
\end{problem}

\subsection*{Algebraic K-theory of rings}

For an $A_\infty$-ring spectrum $R$, the $K$-theory spectrum $K(R)$ is defined as the algebraic $K$-theory\index{algebraic $K$-theory of ring spectra} of the subcategory of compact objects in the $\infty$-category of $R$-modules \cite[VI.3.2]{elmendorf/kriz/mandell/may:1997}. Our goal in this section is to describe what is known about the Taylor towers of the functor $K$ defined in this way.

It is easiest to describe results in the pointed case, i.e. for augmented algebras. Let us fix an $A_\infty$-ring spectrum $R$, and let $\mathscr{A}lg^{\mathrm{aug}}_R$ denote the $\infty$-category of augmented $R$-algebras. We are interested in the Taylor tower (at the terminal object $R$ of $\mathscr{A}lg^{\mathrm{aug}}_R$) of the functor
\[ \tilde{K}_R: \mathscr{A}lg^{\mathrm{aug}}_R \to {\mathscr{S}p} \]
given by $\tilde{K}_R(A) := \operatorname{hofib}(K(A) \to K(R))$.

The first calculation of part of the Taylor tower of $\tilde{K}_R$ was made by Dundas and McCarthy in \cite{dundas/mccarthy:1994}.

\begin{theorem}[Dundas-McCarthy] \label{thm:DM}
Let $R$ be the Eilenberg-MacLane spectrum of a discrete ring. Then
\[ \partial_1\tilde{K}_R \simeq \Sigma \operatorname{THH}(R) \]
the topological Hochschild homology of $R$ of B\"{o}kstedt.
\end{theorem}

To recover the first layer $D_1\tilde{K}_R$ of the Taylor tower from Theorem~\ref{thm:DM}, we need to know the stabilization of the category of augmented $R$-algebras. By work of Basterra and Mandell~\cite{basterra/mandell:2005}, this is equivalent to the category of $R$-bimodules. The suspension spectrum construction takes an augmented $R$-algebra $A$ to its \emph{topological Andr\'{e}-Quillen homology} $\operatorname{taq}_R(A)$, a derived version of $I/I^2$ where $I$ is the augmentation ideal of $A$. We then have
\[ D_1\tilde{K}_R(A) \simeq \Sigma \operatorname{THH}(R;\operatorname{taq}_R(A)) \]
where $\operatorname{THH}(R;M) := R \wedge_{R \wedge R^{op}} M$ is the topological Hochschild homology with coefficients.

Lindenstrauss and McCarthy~\cite{lindenstrauss/mccarthy:2012} have extended Theorem~\ref{thm:DM} to higher layers in the following way. The generalization to all connective ring spectra is due to Pancia~\cite{pancia:2014}.

\begin{theorem}[Lindenstrauss-McCarthy, Pancia] \label{thm:LM}
Let $R$ be a connective ring spectrum. Then for an $R$-bimodule $A$
\[ D_n\tilde{K}_R(A) \simeq \Sigma U^n(R;\operatorname{taq}_R(A))_{hC_n} \]
where $U^n(R;M)$ is a generalization of $\operatorname{THH}(R;M)$ given by the cyclic tensoring of $n$ copies of a bimodule $M$ over $R$, with action of the cyclic group $C_n$ given by permutation.
\end{theorem}

The proof of Theorem~\ref{thm:LM} is a consequence of Lindenstrauss and McCarthy's calculation of the complete Taylor tower of the functor
\[ \tilde{K}_R(T_R(-)) : \operatorname{Bimod}_R \to {\mathscr{S}p} \]
where $T_R: \operatorname{Bimod}_R \to \mathscr{A}lg^{\mathrm{aug}}_R$ is the free tensor $R$-algebra functor given by
\[ T_R(M) := \bigoplus_{n \geq 0} M^{\otimes_R}. \]
It is shown in~\cite{lindenstrauss/mccarthy:2012} that
\[ P_n(\tilde{K}_RT_R)(M) \simeq \operatorname{holim}_{k \leq n} \Sigma U^k(R;M)^{C_k} \]
where the homotopy limit is formed over a diagram of restriction maps
\[ U^k(R;M)^{C_k} \to U^l(R;M)^{C_l} \]
for positive integers $l$ dividing $k$. It is reasonable to expect that more information about the Taylor tower of $\tilde{K}_R$ itself, beyond just the layers, can be extracted from this description, but this has not been done.

One way to approach this question is via the general theory of Section~\ref{sec:operads}. According to Theorem~\ref{thm:operad}, one would expect the derivatives of $\tilde{K}_R$ to be a right module over the derivatives of the identity functor on $\mathscr{A}lg^{\mathrm{aug}}_R$. Theorem~\ref{thm:class} then tells us that the Taylor tower of $\tilde{K}_R$ is determined by the action of a certain comonad on that right module.

Interestingly, and unlike the case for topological spaces, the Taylor tower for the identity functor on $\mathscr{A}lg^{\mathrm{aug}}_R$ is rather easy to describe:
\begin{equation} \label{eq:id} P_nI_{\mathscr{A}lg^{\mathrm{aug}}_R}(A) \simeq I/I^{n+1} \end{equation}
where the right-hand side is a derived version of the quotient of the augmentation ideal $I$ by its $(n+1)$-th power, generalizing the construction of $\operatorname{taq}_R(A) \simeq I/I^2$. This formula seems to have been written down first by Kuhn~\cite{kuhn:2006}, and a more formal version is developed by Pereira~\cite{pereira:2013}. With this calculation it should now be possible to make more progress with the calculation of the Taylor tower of $\tilde{K}_R$.

\begin{remark}
One final remark on the connection between algebraic $K$-theory and Goodwillie calculus is worth making here. Barwick~\cite{barwick:2016} identifies the process of forming algebraic $K$-theory itself as the first layer of a Taylor tower. He defines an $\infty$-category $\mathscr{W}ald_\infty$ of \emph{Waldhausen $\infty$-categories} that are $\infty$-category-theoretic analogues of Waldhausen's categories with cofibrations. He then identifies the algebraic $K$-theory functor as
\[ K \simeq P_1(\iota) \]
where $\iota: \mathscr{W}ald_\infty \to \mathscr{T}op$ is the functor that sends a Waldhausen $\infty$-category to its underlying $\infty$-groupoid of objects. The $1$-excisive property for $K$ is a version of Waldhausen's Additivity Theorem, and this result identifies $K$-theory as the universal example of an additive theory on $\mathscr{W}ald_\infty$ together with a natural transformation from $\iota$.
\end{remark}

\section{Taylor Towers of Infinity-Categories}

Recent work of Heuts has taken Goodwillie calculus in a new direction. In \cite{heuts:2018} he constructs, for each pointed compactly-generated $\infty$-category $\mathscr{C}$, a \emph{Taylor tower of $\infty$-categories}\index{Taylor tower!of an $\infty$-category} for $\mathscr{C}$. This is a sequence of adjunctions
\begin{equation} \label{eq:heuts} \mathscr{C} \rightleftarrows \dots \rightleftarrows \mathscr{P}_n\mathscr{C} \rightleftarrows \mathscr{P}_{n-1}\mathscr{C} \rightleftarrows \dots \rightleftarrows \mathscr{P}_1\mathscr{C} \end{equation}
where $\mathscr{P}_{n}\mathscr{C}$ is a universal approximation to $\mathscr{C}$ by an \emph{$n$-excisive $\infty$-category}.

This approximation has the property that the identity functor on $\mathscr{P}_{n}\mathscr{C}$ is $n$-excisive, and both the unit and counit of the adjunction $\mathscr{C} \rightleftarrows \mathscr{P}_n\mathscr{C}$ are $P_n$-equivalences. In the case $n = 1$, we have $\mathscr{P}_1\mathscr{C} \simeq {\mathscr{S}p}(\mathscr{C})$, the stabilization of $\mathscr{C}$. Thus, the sequence above can be viewed as interpolating between $\mathscr{C}$ and its stabilization via $\infty$-categories that better and better approximate the potentially unstable $\infty$-category $\mathscr{C}$.

Heuts's main results concern the classification of $n$-excisive $\infty$-categories such as $\mathscr{P}_n\mathscr{C}$. Just as in the classification of Taylor towers of functors, this process is broken down into two parts: one \emph{operadic} and the other related to the \emph{Tate} construction.

The first part associates to the $\infty$-category $\mathscr{C}$ the \emph{cooperad} constructed by Lurie that represents the derivatives of the functor $\Sigma^\infty_{\mathscr{C}}\Omega^\infty_{\mathscr{C}}$, as described in Remark~\ref{rem:lurie}. (Recall that this cooperad is actually an $\infty$-\emph{operad} in Lurie's terminology.) The notation for this cooperad in \cite{heuts:2018} is ${\mathscr{S}p}(\mathscr{C})^{\otimes}$, but we will denote it by $\partial_*(\Sigma^\infty_{\mathscr{C}}\Omega^\infty_{\mathscr{C}})$.

\begin{proposition} \label{prop:coalg}
Let $\mathscr{C}$ be a pointed compactly-generated $\infty$-category. Then the suspension spectrum construction for $\mathscr{C}$ can be made into a functor
\[ \Sigma^\infty_{\mathscr{C}} : \mathscr{C} \to \operatorname{Coalg}(\partial_*(\Sigma^\infty_{\mathscr{C}}\Omega^\infty_{\mathscr{C}})) \]
where the right-hand side is the $\infty$-category of non-unital coalgebras over the cooperad $\partial_*(\Sigma^\infty_{\mathscr{C}}\Omega^\infty_{\mathscr{C}})$.
\end{proposition}

\begin{example}
For $\mathscr{C} = \mathscr{T}op_*$, the cooperad $\partial_*(\Sigma^\infty_{\mathscr{C}}\Omega^\infty_{\mathscr{C}})$ is the commutative cooperad in ${\mathscr{S}p}$. In this case the functor in Proposition~\ref{prop:coalg} associates to a based space $X$ the commutative coalgebra structure on $\Sigma^\infty X$ given by the diagonal on $X$.
\end{example}

The category of coalgebras described in Proposition~\ref{prop:coalg} can be thought of as the best approximation to the $\infty$-category $\mathscr{C}$ based on the information in the cooperad $\partial_*(\Sigma^\infty_{\mathscr{C}}\Omega^\infty_{\mathscr{C}})$. The main focus of \cite{heuts:2018} is an understanding of the additional information needed to recover $\mathscr{C}$ itself (or at least its Taylor tower) from this category of coalgebras.

Heuts identifies $\mathscr{P}_n\mathscr{C}$ with an $\infty$-category of \emph{$n$-truncated Tate coalgebras} over the cooperad $\partial_*(\Sigma^\infty_{\mathscr{C}}\Omega^\infty_{\mathscr{C}})$. We do not have space here to provide a precise definition of Tate coalgebra, but we can give the general idea. Full details are in \cite{heuts:2018}.

An $n$-truncated Tate coalgebra over a cooperad $\mathscr{Q}$ consists first of an ordinary (truncated) $\mathscr{Q}$-coalgebra structure on an object $E$ in $\mathscr{P}_1\mathscr{C} \simeq {\mathscr{S}p}(\mathscr{C})$. That is, for $k \leq n$, we have structure maps of the form
\[ c_k: E \to [\mathscr{Q}(k) \wedge E^{\wedge k}]^{h\Sigma_k}. \]
Each map $c_k$ is then required to be compatible (via the canonical map from fixed points to the Tate construction) with a certain natural transformation
\[ t_k: E \to [\mathscr{Q}(k) \wedge E^{\wedge k}]^{t\Sigma_k}, \]
that Heuts calls the \emph{Tate diagonal} for $\mathscr{C}$.\index{Tate diagonal}

Crucially, the natural transformation $t_k$ is defined only when $E$ is already a $(k-1)$-truncated Tate $\mathscr{Q}$-coalgebra, and $t_k$ must be compatible, via the cooperad structure on $\mathscr{Q}$, with the maps $c_1,\dots,c_{k-1}$. Also note that the Tate diagonal $t_k$ depends on the $\infty$-category $\mathscr{C}$ and not just the cooperad $\mathscr{Q}$. It is the choice of these maps that carries the extra information needed to reconstruct the Taylor tower of $\mathscr{C}$.

We can now state Heuts's main result from \cite{heuts:2018}.

\begin{theorem} \label{thm:heuts-class}
Let $\mathscr{C}$ be a pointed compactly-generated $\infty$-category with $\mathscr{Q} := \partial_*(\Sigma^\infty_{\mathscr{C}}\Omega^\infty_{\mathscr{C}})$, the corresponding cooperad in the stable $\infty$-category ${\mathscr{S}p}(\mathscr{C})$. Then there is a unique sequence of Tate diagonals, i.e. natural transformations
\[ t_k: E \to [\mathscr{Q}(k) \wedge E^{\wedge k}]^{t\Sigma_k} \]
where $t_k$ is defined for $E \in \mathscr{P}_{k-1}\mathscr{C}$, such that $\mathscr{P}_k\mathscr{C}$ is equivalent to the ind-completion of the $\infty$-category of $k$-truncated Tate $\mathscr{Q}$-coalgebras whose underlying object of ${\mathscr{S}p}(\mathscr{C})$ is compact.
\end{theorem}

Note that this is intrinsically an inductive result: the construction of $\mathscr{P}_{k-1}\mathscr{C}$ needs to be made in order to understand the Tate diagonal $t_k$ that allows for the definition of $\mathscr{P}_k\mathscr{C}$. Here is how this process plays out for based spaces.

\begin{example}
We have $\mathscr{P}_1\mathscr{T}op_* \simeq {\mathscr{S}p}$ and $\partial_*(\Sigma^\infty\Omega^\infty) \simeq \mathscr{C}om$, the commutative cooperad of spectra with $\mathscr{C}om(n) \simeq S^0$ for all $n$. The first Tate diagonal has no compatibility requirements and is simply a natural transformation
\[ t_2: Y \to (Y \wedge Y)^{t\Sigma_2} \]
defined for all $Y \in {\mathscr{S}p}$.

Since both the source and target are $1$-excisive functors, such a natural transformation is determined by its value on the sphere spectrum, where it takes the form of a map
\[ t_2: S^0 \to (S^0)^{t\Sigma_2} \simeq (S^0)^{\wedge}_2 \]
whose target is the $2$-complete sphere by the Segal Conjecture~\cite{carlsson:1984}. The specific Tate diagonal that corresponds to the $2$-excisive $\infty$-category $\mathscr{P}_2\mathscr{T}op_*$ is then the ordinary $2$-completion map $S^0 \to (S^0)^{\wedge}_2$.

According to Theorem~\ref{thm:heuts-class}, a (compact) object of $\mathscr{P}_2\mathscr{T}op_*$ consists of a finite spectrum $Y$ with map
\[ c_2: Y \to (Y \wedge Y)^{h\Sigma_2} \]
and a homotopy between $t_2$ and the composite
\[ Y \to (Y \wedge Y)^{h\Sigma_2} \to (Y \wedge Y)^{t\Sigma_2}. \]

Now consider a Tate diagonal for $\mathscr{P}_2\mathscr{T}op_*$. This should be a natural transformation
\[ t_3: Y \to (Y^{\wedge 3})^{t\Sigma_3} \]
that is compatible with the map
\[ Y \to \left[ \prod_3 Y \wedge (Y \wedge Y) \right]^{t\Sigma_3} \]
that is built from iterating $c_2$, where the three copies of $Y \wedge (Y \wedge Y)$ are indexed by the three binary trees with leaves labelled $\{1,2,3\}$, and composing with the map from fixed points to the Tate spectrum. As in the $n = 2$ case, there is a specific such natural transformation $t_3$ that corresponds to $\mathscr{P}_3\mathscr{T}op_*$ though it doesn't seem to be as easy to describe explicitly for general $Y \in \mathscr{P}_2\mathscr{T}op_*$. (For objects $Y$ that are of the form $\Sigma^\infty X$ for a space $X$, $t_3$ is induced by the diagonal on $X$.)

An object of $\mathscr{P}_3\mathscr{T}op_*$ then consists of an object of $\mathscr{P}_2\mathscr{T}op_*$ together with a map
\[ c_3: Y \to (Y^{\wedge 3})^{h\Sigma_3} \]
that lifts $t_3$ and is compatible with $c_2$.
\end{example}

As in the case of Taylor towers of functors, a substantial simplification occurs when the Tate data vanishes, in which case the Taylor tower of the $\infty$-category $\mathscr{C}$ is completely determined by the cooperad $\partial_*(\Sigma^\infty_{\mathscr{C}}\Omega^\infty_{\mathscr{C}})$.

\begin{corollary}[{\cite[6.14]{heuts:2018}}]
Let $\mathscr{C}$ be an $\infty$-category such that for any object $X \in {\mathscr{S}p}(\mathscr{C})$ with $\Sigma_k$-action, the Tate construction $X^{t\Sigma_k}$ is trivial, for all $k \geq 2$. Then $\mathscr{P}_n\mathscr{C}$ is equivalent to the ind-completion of the category of $n$-truncated $\partial_*(\Sigma^\infty_{\mathscr{C}}\Omega^\infty_{\mathscr{C}})$-coalgebras whose underlying object is compact in ${\mathscr{S}p}(\mathscr{C})$.
\end{corollary}

\begin{remark}
The Tate diagonals for an $\infty$-category $\mathscr{C}$ bear a close relationship to the Taylor tower of the functor $\Sigma^\infty_{\mathscr{C}}\Omega^\infty_{\mathscr{C}}: {\mathscr{S}p}(\mathscr{C}) \to {\mathscr{S}p}(\mathscr{C})$. In particular $t_n$ can be written as a composite
\[ E \to P_{n-1}(\Sigma^\infty_{\mathscr{C}}\Omega^\infty_{\mathscr{C}})(E) \arrow{e,t}{t'_n} [\partial_n(\Sigma^\infty_{\mathscr{C}}\Omega^\infty_{\mathscr{C}}) \wedge E^{\wedge n}]^{t\Sigma_n} \]
where the second map here is the bottom horizontal map in McCarthy's square (\ref{thm:McCarthy}), and the first map contains the structure of an $(n-1)$-truncated Tate coalgebra on $E$. The map on derivatives induced by $t'_n$ can furthermore be identified with the coalgebra structure from Theorem~\ref{thm:class} that classifies the Taylor tower of the functor $\Sigma^\infty_{\mathscr{C}}\Omega^\infty_{\mathscr{C}}$.
\end{remark}

\section{The Manifold and Orthogonal Calculi}

While this article has been focused on Goodwillie's calculus of homotopy functors, there are two other theories of ``calculus'', developed by Michael Weiss, that are inspired by, and related to, Goodwillie calculus to varying degrees. They are called {\it manifold calculus} (initially known as embedding calculus) and {\it orthogonal calculus}. In this section we give a review of these theories and some applications they have had.

\subsection*{Manifold calculus}\index{manifold calculus}
Manifold calculus was initially developed in~\cite{weiss:1999, goodwillie/weiss:1999}, see also~\cite{boavida/weiss:2013}. It concerns {\it contravariant} functors on manifolds. Of all the brands of calculus, this one is closest to classical sheaf theory, and therefore may look the most familiar.

Suppose $M$ is an $m$-manifold. Let $F$ be a presheaf of spaces on $M$. In other words, $F$ is a contravariant functor $F\colon \mathcal O(M)\to\mathscr{T}op$, where $\mathcal{O}(M)$ is the poset of open subsets of $M$.\footnote{For concreteness we focus on presheaves in $\mathscr{T}op$, but it seems likely that the theory can be extended without difficulty to presheaves in a general $(\infty, 1)$-category.} We assume that $F$ takes isotopy equivalences to homotopy equivalences, and filtered colimits to inverse limits. The motivating example is the presheaf $U\mapsto \operatorname{Emb}(U, N)$, where $N$ is a fixed manifold, and $\operatorname{Emb}(-, -)$ denotes the space of smooth embeddings.

The notion of excisive functor in  manifold calculus is analogous to that in Goodwillie's homotopy calculus. We say that a presheaf is \emph{$n$-excisive}\index{excisive functor} if it takes strongly cocartesian $(n+1)$-cubes in $\mathcal{O}(M)^{op}$ to cartesian cubes in $\mathscr{T}op$. For example, $F$ is 1-excisive if for any open sets $U, V\subset M$, the following diagram is a (homotopy) pull-back square
\[ \begin{diagram}
  \node{F(U\cup V)} \arrow{s} \arrow{e} \node{F(V)} \arrow{s} \\
  \node{F(U)} \arrow{e} \node{F(U\cap V)}
\end{diagram} \]
Thus we see that 1-excisive presheaves are, essentially, sheaves, except that the sheaf condition has to be interpreted in a homotopy invariant way. Sometimes 1-excisive functors are called homotopy sheaves.

Similarly, $n$-excisive functors can be interpreted as homotopy sheaves with respect to a different Grothendieck topology, where one says that $V_1, \ldots, V_k$ cover $U$ if $U^n=V_1^n\cup\cdots\cup V_k^n$.

Just as in Goodwillie calculus, the inclusion of the category of $n$-excisive presheaves into the category of all presheaves has a left adjoint. The adjoint is constructed by a process of sheafification, rather than stabilization. There is another, slightly different procedure for constructing the approximation that is often useful. We will describe it now.

\begin{definition}
For an $m$-manifold $U$, Let $\mathcal O_n(U)\subset \mathcal O(U)$ be the poset of open subsets of $U$ that are diffeomorphic to a disjoint union of at most $n$ copies of $\mathbb R^m$. Given a presheaf $F$ on $M$, define a new presheaf $T_nF$ by the formula
\[
T_nF(U)=\operatorname{holim}_{V\in \mathcal O_n(U)} F(V)
\]
\end{definition}
\begin{theorem}[Weiss~\cite{weiss:1999}]
Assume that $F$ takes isotopy equivalences to homotopy equivalences and filtered colimits to inverse homotopy limits. Then $T_nF$ is $n$-excisive, and the natural transformation $F\to T_nF$ is initial among all natural transformations from $F$ to an $n$-excisive functor.
\end{theorem}
Thus $T_nF$ is the universal $n$-excisive approximation of $F$. As in homotopy calculus, there are natural transformations $T_nF\to T_{n-1}F$. So the approximations fit into a ``tower'' of functors under $F$
\[ F \to \cdots \to T_nF \to T_{n-1}F \to \cdots \]
sometimes called the \emph{embedding tower}\index{embedding tower} when $F$ is $\operatorname{Emb}(-,N)$ for a manifold $N$.

Continuing the analogy with homotopy calculus, there is a classification theorem for homogeneous functors that can be used to describe the layers in the tower of approximations. Unlike in homotopy calculus, a homogeneous functor is not classified by a spectrum, but by a fibration over a space of configurations of points in $M$. We will give below a description of the homotopy fibre of the map $T_nF\to T_{n-1}F$ (for space-valued $F$).

Let ${M \choose n}$ be the space of unordered $n$-tuples of pairwise distinct points of $M$. Given a point ${\mathbf x}=[x_1, \ldots, x_n]$ of ${M \choose n}$, let $U_{\mathbf x}$ be a tubular neighborhood of $\{x_1, \ldots, x_n\}$ in $M$. In particular, $U_{\mathbf  x}=\coprod_{i=1}^n U_i$ is diffeomorphic to a disjoint union of $n$ copies of $\mathbb R^n$. We have an $n$-dimensional cubical diagram, indexed by subsets of $\{1,\ldots, n\}$ and opposites of inclusions, that sends a subset $S\subset \{1,\ldots, n\}$ to $F(\coprod_{i\in S} U_i)$.
Let $\widehat F(U_{\mathbf  x})$ be the total homotopy fibre of this diagram. One can naturally construct a fibred space over ${M \choose n}$ where the fibre at ${\mathbf x}$ is equivalent to $\widehat F(U_{\mathbf  x})$.
\begin{theorem}[Weiss \cite{weiss:1999}]\label{theorem: layer}
A choice of point in $T_{n-1}F(M)$ gives a germ of a section of this fibration near the fat diagonal. The homotopy fibre of $T_nF(M)\to T_{n-1}F(M)$ is equivalent to the space of sections of the fibration, that agree near the fat diagonal with the local section prescribed by the choice of point in $T_{n-1}F(M)$.
\end{theorem}
\begin{remark}
Suppose $f\colon \mathbb R \to \mathbb R$ is a function. For each $n\ge 0$, let $t_nf$ be the unique polynomial of degree $n$ for which $f(i)=t_nf(i)$ for $i=0, 1, \ldots, n$. There is a well known formula for $t_nf$. Namely $t_nf(m)=\sum_{i=0}^n \hat f(i)\cdot {m \choose i}$. Here $\hat f(i)=\sum_{j=0}^i(-1)^{j}f(i-j)$ is the $i$-th cross-effect of $f$. In particular, the $n$-th term of $t_nf(m)$ is $\hat f(n)\cdot {m \choose n}$. Note the formal similarity of this formula with the formula for the $n$-th layer in the manifold calculus tower in Theorem~\ref{theorem: layer}. This suggests that the construction $T_nF$ of manifold calculus is analogous to the interpolating polynomial $t_nf$ rather than to the Taylor polynomial of $F$. In fact, it is true that the map $F\to T_nF$ is characterized by  the property that it is an equivalence on objects of $\mathcal O_n(M)$. This confirms the intuition that the approximations in manifold calculus are analogous to interpolation polynomials.
\end{remark}

Just as in Goodwillie calculus, the question of convergence is important, and there is a general theory of analytic functors, for which the tower of approximations converges strongly. The main result on the subject is the following deep theorem of Goodwillie and Klein. Let $\operatorname{Emb}$ denote the space of smooth embeddings.
\begin{theorem}[\cite{goodwillie/klein:2015}]
Let $M^m$, $N^d$ be smooth manifolds. Consider the presheaf on $M$ given by the formula $U\mapsto \operatorname{Emb}(U, N)$. The map $\operatorname{Emb}(M, N)\to T_n\operatorname{Emb}(M, N)$ is $(d-m-2)n-m+1$-connected.
\end{theorem}
The theorem says in particular that if $d-m\ge 3$, then the connectivity of the map goes to infinity linearly in $n$, and the tower converges strongly. Even in situations where there is no strong convergence, the tower may be useful for constructing invariants of embeddings and obstructions to existence of embeddings. For example, in the case of $m=1$, $d=3$, there is a close connection between the manifold tower and finite type knot invariants~\cite{volic:2006}.

It is often possible to express the tower $T_kF$ in terms of mapping spaces between modules over the little discs operad. This idea probably first appeared in~\cite{sinha:2006} in the context of spaces of knots, and then was developed in~\cite{arone/lambrechts/volic:2007, arone/turchin:2014, turchin:2013}, and finally and definitively in~\cite{boavida/weiss:2013}. Thanks to this connection, facts about the homotopy of the little disc operads (most notably Kontsevich's formality theorem) have consequences regarding the homotopy type of spaces of embeddings~\cite{arone/lambrechts/volic:2007, arone/lambrechts/turchin/volic:2008, arone/turchin:2015}.

Recently manifold calculus was used by Michael Weiss in the course of proving some striking results about Pontryagin classes of topological bundles~\cite{weiss:2015}.

\subsection*{Orthogonal calculus}\index{orthogonal calculus} Let $\mathcal{J}$ be the category of finite-dimensional Euclidean spaces and linear isometric inclusions. Orthogonal calculus~\cite{weiss:1995} concerns continuous functors from $\mathcal{J}$ to the category of topological spaces. It seems likely that topological spaces can be replaced in a routine way with a more general topologically enriched category. The paper~\cite{barnes/oman:2013} develops orthogonal calculus using the perspective of Quillen model categories. Formally, orthogonal calculus is more similar to homotopy calculus than manifold calculus is, and the paper~\cite{barnes/eldred:2016} explores this similarity in some depth. Orthogonal calculus has probably received the least attention of the three main brands, and we believe it is ripe for exploration.

Let $V$ denote a generic Euclidean space. Here are some examples of functors to which one might profitably apply orthogonal calculus: $V\mapsto BO(V)$, $V\mapsto G(V)$ (the space of homotopy self equivalences of the unit sphere of $V$), $V\mapsto BTop(V)$, $V\mapsto \operatorname{Emb}(M, N\times V)$, etc.

Such functors tend to be easier to understand when evaluated at high dimensional vector spaces. The rough idea of orthogonal calculus is to analyze the behavior of a functor on high-dimensional vector spaces and then to extrapolate to low-dimensions. As a result one obtains a kind of Taylor expansion at infinity. Thus to a functor $F\colon \mathcal{J} \to \mathscr{T}op$ one associates a tower of functors under $F$
\[
F\to \cdots  \to P_nF \to P_{n-1}F \to \cdots \to P_0F
\]
where $P_nF$ plays the role of the Taylor tower of $F$ at infinity. In particular, the constant term $P_0F$ is equivalent to $\operatorname{colim}_{n\to \infty} F(\mathbb R^n)$. The higher Taylor approximations are usually difficult to describe explicitly, 
but the layers in the tower can be described in terms of certain spectra that play the role of derivatives. This is similar to homotopy calculus, but unlike in homotopy calculus, the sequence of derivatives of a functor in orthogonal calculus is an {\it orthogonal} sequence of spectra, rather than a symmetric one. This means that the $n$-th derivative is a spectrum with an action of $O(n)$. The following theorem summarizes some results from~\cite{weiss:1995} about homogeneous functors and layers in orthogonal calculus. Compare with Theorem~\ref{thm:hom} and Example~\ref{example:hom}
\begin{theorem}
To a continuous functor $F\colon \mathcal{J} \to \mathscr{T}op$ one can associate an orthogonal sequence of spectra, $\partial_1 F, \ldots, \partial_n F, \ldots$. The homotopy fibre of the map $P_nF(V)\to P_{n-1}F(V)$ is naturally equivalent to the functor that sends $V$ to $\Omega^\infty\left( \partial_n F \wedge S^{nV}\right)_{hO(n)}$. Here $S^{nV}$ is the one-point compactification of the vector space $nV$.
\end{theorem}
\begin{examples}
Recall that $\mathcal J$ is the category of Euclidean spaces and linear isometric inclusions. Suppose $V_0, V$ are Euclidean spaces. Let $\operatorname{Hom}_{\mathcal J}(V_0, V)$ denote the space of linear isometric inclusions from $V_0$ to $V$. Now fix $V_0$ and consider the functor $\Omega^\infty \Sigma^\infty \operatorname{Hom}_{\mathcal J} (V_0, -)_+$. The Taylor tower of this functor was described in~\cite{arone:2002}. This is a rare example of a functor whose Taylor tower can be described rather explicitly. When $\dim(V_0)\le \dim(V)$, the Taylor tower splits as a product of its layers, and in fact this splitting is equivalent to a classical stable splitting of Stiefel manifolds discovered by Haynes Miller~\cite{miller:1985, arone:2001}.

For another example, consider the functor $V\mapsto BO(V)$. The $n$-th derivative of this functor can be described in terms of the complex of direct-sum decompositions of $\mathbb R^n$. A (proper) direct-sum decomposition is a collection (at least two) of pairwise-orthogonal non-zero subspaces of $\mathbb R^n$ whose direct sum of $\mathbb R^n$. We can partially order such decompositions by refinement and obtain a topological poset whose realization we denote by $L_n$. Let $L_n^\diamond$ be the unreduced suspension of that realization. The space $L_n^\diamond$ is analogous to the complex of partitions $T_n$ from definition~\ref{def:tn}. The $n$-th derivative of the functor $BO(V)$ is equivalent to $\mathbb{D}(L_n^\diamond)\wedge S^{ad_n}$, where $S^{ad_n}$ is the one-point compactification of the adjoint representation of $O(n)$. Compare with Theorem~\ref{theorem: brenda} about the Goodwillie derivatives of the identity functor.  The space $L_n^\diamond$, and even more so its complex analogue, has many striking properties. It appeared in several recent works in an interesting way. For example, it was used in the study of the stable rank filtration of complex $K$-theory~\cite{arone/lesh:2007,arone/lesh:2010}, and in the study of the Balmer spectrum of the equivariant stable homotopy category~\cite{barthel/hausmann/naumann/nikolaus/noel/stapleton:2017}.
\end{examples}

It would be especially interesting to understand the derivatives of the functor $V\mapsto \operatorname{Emb}(M, N\times V)$. We speculate that the $n$-th derivative of this functor can be expressed in terms of moduli space of connected graphs with points marked in $M$, for which the quotient space of the graph by the subset of marked points is homotopy equivalent to a wedge of $n$ circles.

\section{Further Directions}

The basic concepts of Goodwillie calculus are very general and can be applied to a wide variety of homotopy-theoretic settings. Nonetheless, not many calculations have been done outside of representable functors, the identity functor and algebraic $K$-theory. We have seen that the Taylor tower of the identity functor plays a key role in the theory, so a calculation of that Taylor tower, or at least its layers, would be valuable in other contexts. Obvious candidates for further exploration include:
\begin{itemize}
  \item unpointed and parameterized homotopy theory: that is, a more detailed understanding of Taylor towers at base objects other than the one-point space;
  \item unstable equivariant homotopy theory: presumably the derivatives of the identity here are some equivariant version of the spectral Lie operad;
  \item unstable motivic homotopy theory: as far as we know, no work has been done in this direction despite the wealth of structure the Taylor tower of the identity should possess in this case.
\end{itemize}
In each of these cases calculations of the derivatives of the identity, their operad structure and Heuts's Tate diagonals, could reveal something deep about how unstable information is built from stable.

The equivariant setting is potentially very interesting. Dotto~\cite{dotto:2017} has generalized Goodwillie calculus to a $G$-equivariant context, for a finite group $G$, in which the Taylor tower of a functor is replaced by a diagram of approximations indexed by the finite $G$-sets. This is based on an idea of Blumberg~\cite{blumberg:2006} for a $G$-equivariant version of excision.

Dotto's approach seems to be better able to make use of the power of modern equivariant stable homotopy theory, something rather neglected by Goodwillie's original version. In particular, the derivatives of a functor in this setting are genuine $G$-spectra. Dotto calculates the derivatives of the identity functor on pointed $G$-spaces as equivariant Spanier-Whitehead duals of the partition poset complexes of \cite{arone/mahowald:1999}. It seems reasonable to expect that much of the theory described in this article could be extended the equivariant setting, but little has been done yet.

Manifold and orthogonal calculus have been applied to problems in geometric topology, but there surely is much that remains to be done. In particular orthogonal calculus should have a lot of potential that has barely began to be tapped. We mentioned above that it would be interesting to apply orthogonal calculus to the study of embedding spaces by considering the functor $V\mapsto \operatorname{Emb}(M, N\times V)$. An even more interesting and challenging example is given by the functor $V\mapsto \mbox{BTop}(V)$ and the closely related functor  $V\mapsto B^{V+1} \mbox{Diff}_c(V)$. Here $B^{V+1}$ stands for $V+1$-fold delooping, and Diff$_c(V)$ stands for the space of diffeomorphisms of $V$ that equal the identity outside a compact set. We speculate that the $n$-th derivative of these functors is related to the moduli space of graphs homotopy equivalent to a wedge of $n$ circles (without marked points).

The close analogies that connect Goodwillie calculus with the manifold and orthogonal versions suggests the existence of a more encompassing framework that could also provide new instances of ``calculus'' for other kinds of functors. 
One such framework is the theory of tangent categories of Rosick\'{y}~\cite{rosicky:1984} and Cockett and Cruttwell~\cite{cockett/cruttwell:2014}, developed to axiomatize the structure of the category of manifolds and smooth maps. Lurie's notion of tangent bundle for an $\infty$-category~\cite[7.3.1]{lurie:2017} fits Goodwillie calculus into this story, and opens up a way to make precise ideas of Goodwillie on applying differential-geometric ideas such as connections and curvature directly to homotopy theory. There is a helpful intuition that stable $\infty$-categories such as spectra are ``flat'', whereas unstable worlds are ``curved''. However, manifold and orthogonal calculus do not appear to fit into this same picture, so perhaps an even more general theory is still waiting to be discovered.


\bibliographystyle{amsplain}
\bibliography{mcching}

\end{document}